\documentclass[final,1p,times]{elsarticle}

\usepackage{amsthm,amsmath,amssymb,amsfonts,graphicx,graphics,latexsym,exscale,cmmib57,dsfont,bbm,amscd,ulem}
\usepackage{bbm}
\usepackage{mathrsfs}
\usepackage{color,soul}
\usepackage[colorlinks=true]{hyperref}

\newcommand{\bP}{\boldsymbol{P}}
\newcommand{\p}{\boldsymbol{p}}
\newcommand{\A}{\boldsymbol{A}}
\newcommand{\B}{\boldsymbol{B}}
\newcommand{\blambda}{\boldsymbol{\lambda}}
\newcommand{\bK}{\boldsymbol{K}}

\newcommand{\bC}{\boldsymbol{C}}
\newcommand{\bS}{\boldsymbol{S}}
\newcommand{\bX}{\boldsymbol{X}}
\newcommand{\bY}{\boldsymbol{Y}}
\newcommand{\bZ}{\boldsymbol{Z}}

\newcommand{\bgamma}{\boldsymbol{\gamma}}
\newcommand{\bd}{\boldsymbol{d}}
\newcommand{\ba}{\boldsymbol{a}}

\newcommand{\w}{\boldsymbol{w}}
\newcommand{\bta}{\boldsymbol{\eta}}
\newcommand{\bT}{\boldsymbol{T}}
\newcommand{\bcdot}{\boldsymbol{\cdot}}

\newcommand{\bnu}{\boldsymbol{\nu}}
\newcommand{\brho}{\boldsymbol{\rho}}
\newcommand{\bxi}{\boldsymbol{\xi}}

\newcommand{\bea}{\begin{eqnarray}}
\newcommand{\eea}{\end{eqnarray}}
\newcommand{\bean}{\begin{eqnarray*}}
\newcommand{\eean}{\end{eqnarray*}}
\newtheorem*{conj}{Spectral finiteness}
\newtheorem*{conjecture}{Conjecture}

\newtheorem*{Dirichlet}{Dirichlet's Theorem}

\newtheorem{theorem}{Theorem}[section]
\newtheorem{lemma}[theorem]{Lemma}
\newtheorem{cor}[theorem]{Corollary}

\theoremstyle{definition}
\newtheorem{defn}[theorem]{Definition}

\newtheorem{remark}[theorem]{Remark}

\numberwithin{equation}{section}

\journal{The Journal of the Franklin Institute}

\begin{document}

\begin{frontmatter}

\title{Robust periodic stability implies uniform exponential stability of Markovian jump linear systems and random linear ordinary differential equations}

\author{Xiongping Dai}
\ead{xpdai@nju.edu.cn}
\address{Department of Mathematics, Nanjing University, Nanjing 210093, People's Republic of China}

\begin{abstract}
In this paper, we mainly show the following two statements.
\begin{enumerate}
\item[(1)] A discrete-time Markovian jump linear system is uniformly exponentially stable if and only if it is robustly periodically stable, by using a Gel'fand-Berger-Wang formula proved here.
\item[(2)] A random linear ODE driven by a semiflow with closing by periodic orbits property is uniformly exponentially stable if and only if it is robustly periodically stable, by using Shantao Liao's perturbation technique and the semi-uniform ergodic theorems.
\end{enumerate}
Our proofs involve ergodic theory in both of the above two cases. In addition, counterexamples are constructed to the robustness condition and to spectral finiteness of linear cocycle.
\end{abstract}

\begin{keyword}
Markovian jump linear system\sep random linear ODE\sep exponential stability\sep robust periodic stability\sep Gel'fand-Berger-Wang formula\sep spectral finiteness.

\medskip
\MSC[2010] Primary 37H05\sep 34D23; Secondary 93D09\sep 93C30\sep 15A18\sep 37N35.
\end{keyword}

\end{frontmatter}

\section{Introduction}\label{sec1}%

In this paper, we study the absolute/uniform exponential stability of a discrete-time Markovian jump linear system and a random linear ordinary differential equation driven by a semiflow with closing by periodic orbits property, using ergodic theory and Shantao Liao's perturbation technique developed in the differentiable dynamical systems.

Let $d\ge1$ be an arbitrary integer, $\|\cdot\|$ an arbitrarily given vector norm on $\mathbb{R}^d$, $\mathbb{N}=\{1,2,\dotsc\}$, and $\mathbb{Z}=\{0,\pm1,\pm2,\dotsc\}$. Throughout this paper, by $\mathbb{R}^{d\times d}$ (resp.~$\mathrm{GL}(d,\mathbb{R})$) we denote the usual topological spaces of all real $d$-by-$d$ matrices (resp. nonsingular) with the usual matrix norm $\pmb{\|}\cdot\pmb{\|}$ compatible with $\|\cdot\|$ on $\mathbb{R}^d$. By $\mathrm{C}(X,\mathbb{R}^{d\times d})$ (resp. $\mathrm{C}(X,\mathrm{GL}(d,\mathbb{R})$), we mean the space of all continuous functions from a compact metric space $X$ into $\mathbb{R}^{d\times d}$ (resp.~$\mathrm{GL}(d,\mathbb{R})$) endowed with the uniform convergence topology. Then $f_n\to f$ in $\mathrm{C}(X,\mathbb{R}^{d\times d})$ (resp.~$\mathrm{C}(X,\mathrm{GL}(d,\mathbb{R}))$) if and only if $\sup_{x\in X}\pmb{\|}f_n(x)-f(x)\pmb{\|}\to 0$.

\subsection{Markovian jump linear system}\label{sec1.1}
Given an arbitrary integer $K\ge2$, let $\A=\{A_1,\dotsc,A_K\}\subset\mathbb{R}^{d\times d}$ be an arbitrarily given ordered set thought of as a discrete topological space. Let $(\Omega,\mathscr{F},\mathbb{P})$ be a probability space and $\bxi=\big{(}\xi_n\colon\Omega\rightarrow\A\big{)}_{n\ge1}$ be a time-homogeneous Markovian chain with state space $\A$ and obeying the Markov transition probability matrix
\begin{equation*}
\bP=(p_{ij})_{K\times K}\quad (\textrm{i.e. } \mathbb{P}\{\xi_{n+1}=j\,|\,\xi_n=i\}=p_{ij},\ 1\le i,j\le K,\quad\forall n\ge1).
\end{equation*}
Let there be fixed an initial probability distribution
$$\p=(p_1,\dotsc,p_K)\quad (\textrm{i.e. } \mathbb{P}\{\xi_1=A_k\}=p_k\textrm{ for }1\le k\le K)$$
such that $\p\bP=\p$. We note that such $\p$ is always existent from the Perron-Frobenius theorem.

Let $\mathbb{A}=(\mathbb{A}_{ij})_{K\times K}$ be the matrix of zeros and ones defined by: $\mathbb{A}_{ij}=1$ if and only if $p_{ij}>0$, and
$\mathbb{A}_{ij}=0$ if and only if $p_{ij}=0$. By $\varSigma_\mathbb{A}^+$, we denote the standard one-sided symbolic space of finite-type consisting of all the one-sided infinite symbolic sequences
$\sigma\colon\mathbb{N}\rightarrow\{1,\dotsc,K\}$
satisfying the $\mathbb{A}$-constraint: $\mathbb{A}_{\sigma(n)\sigma(n+1)}=1$ for all $n\ge1$. We can obtain a natural probability distribution, called the ``$(\p,\bP)$-Markovian measure" and simply write as $\mu_{\p,\bP}$, on $\varSigma_\mathbb{A}^+$, which is such that for any $n\ge1$,
\begin{equation*}
\mu_{\p,\bP}([i_1,\dotsc,i_n]_{\mathbb{A}})=\begin{cases}
p_{i_1}& \textrm{if }n=1;\\
p_{i_1}p_{i_1i_2}\dotsm p_{i_{n-1}i_n}& \textrm{if }n\ge2,
\end{cases}
\end{equation*}
for all cylinder sets $[i_1,\dotsc,i_n]_{\mathbb{A}}:=\{\sigma\in\varSigma_\mathbb{A}^+\colon \sigma(1)=i_1,\dotsc,\sigma(n)=i_n\}\subset\varSigma_\mathbb{A}^+$.
Based on $\A$, this then naturally induces a discrete-time linear switching dynamical system:
\begin{equation}\label{eq1.1}
x_n=A_{\sigma(n)} x_{n-1},\quad x_0\in\mathbb{R}^d, n\ge1 \textrm{ and }\sigma\in\varSigma_\mathbb{A}^+.
\end{equation}
We can see that the Markovian chain $\bxi$ is $\mathbb{P}$-almost surely stable if and only if (\ref{eq1.1}) is $\mu_{\p,\bP}$-almost surely stable; see, e.g., \cite{DHX}.
Then we may identify the Markovian chain $\bxi$ on $(\Omega,\mathscr{F},\mathbb{P})$ valued in $\A$ with (\ref{eq1.1}) and call the later a Markovian jump linear system (MJLS for short).

From now on we fix such an $(0,1)$-matrix $\mathbb{A}=(\mathbb{A}_{ij})_{K\times K}$. The above MJLS $(\ref{eq1.1})$ is said to be \textit{absolutely exponentially stable}, provided that for every initial state $x_0\in\mathbb{R}^d$ and any switching law $\sigma\in\varSigma_\mathbb{A}^+$, the corresponding state orbit $\{x_n(x_0,\sigma)\}_{n=1}^\infty$ of $(\ref{eq1.1})$ is such that
\begin{equation*}
\limsup_{n\to\infty}\frac{1}{n}\log\|x_n(x_0,\sigma)\|<0.
\end{equation*}
Equivalently, one can find two constants $C>0$ and $0<\gamma<1$ so that
\begin{equation*}
\pmb{\|}A_{\sigma(n)}\dotsm A_{\sigma(1)}\pmb{\|}\le C\gamma^n\quad\forall n\ge1\textrm{ and }\sigma\in\varSigma_\mathbb{A}^+;
\end{equation*}
namely, the MJLS $(\ref{eq1.1})$ is \textit{uniformly exponentially stable}. This stability does not depend on the vector-norm $\|\cdot\|$ on $\mathbb{R}^d$ and its induced matrix-norm $\pmb{\|}\cdot\pmb{\|}$ on $\mathbb{R}^{d\times d}$ used here.

The stability issues, more precisely \textit{how to characterize the stability}, of such linear dynamical systems are very important from both the theoretical and practical viewpoints and have drawn a lot of attentions in the recent years; see, e.g., \cite{LM, Mic, Mar, SWMWK} and books \cite{Lib, SG}.

A base point $\sigma\in\varSigma_\mathbb{A}^+$ is called a \textit{periodic switching law of period $\pi$}, for some $\pi\ge1$, if it is of the form
$\sigma=(\uwave{i_1,\dotsc,i_{\pi}},\uwave{i_1,\dotsc,i_{\pi}},\dotsc)$,
where the constituent word $(i_1,\dotsc,i_{\pi})$, called the \textit{generator} of $\sigma$, is such that $\mathbb{A}_{i_ki_{k+1}}=1$ for $1\le k<\pi$ and $\mathbb{A}_{i_\pi i_1}=1$. It is completely determined by $\mathbb{A}$ and independent of $\A$.
For any $n\ge1$, by
$W_\textrm{per}^n(\mathbb{A})$ we mean the set of all $n$-length words $w=(w_1,\dotsc,w_n)$ in $\{1,\dotsc,K\}^n$ such that $\mathbb{A}_{w_kw_{k+1}}=1$ for $1\le k<n$ and $\mathbb{A}_{w_nw_1}=1$. Clearly, for any periodic switching law of period $\pi$ there exists a corresponding generator $w\in W_\textrm{per}^\pi(\mathbb{A})$ and vice versa.

For any square matrix $B\in\mathbb{R}^{d\times d}$, by $\rho(B)$ we mean the usual spectral radius of $B$; that is, if $\lambda_1,\dotsc,\lambda_r$ are all the distinct eigenvalues of $B$ then $\rho(B)=\max_{1\le i\le r}|\lambda_i|$.

It is a well-known fact that the judgement of the stability of (\ref{eq1.1}) governed by an aperiodic switching law $\sigma$ is by no means a trivial task in general; however, if $\sigma$ is periodic with the generator $(i_1,\dotsc,i_\pi)$, then this task becomes trivial at once by only computing $\rho(A_{i_\pi}\dotsm A_{i_1})$ because (\ref{eq1.1}) governed by $\sigma$ is stable if and only if $\rho(A_{i_\pi}\dotsm A_{i_1})<1$ from the well-known Gel'fand formula of a matrix
\begin{equation}\label{eq1.2}
\rho(A)=\lim_{n\to\infty}\sqrt[n]{\pmb{\|}A^n\pmb{\|}}\quad\forall A\in\mathbb{R}^{d\times d}.
\end{equation}
Unfortunately, the periodic stability of (\ref{eq1.1}) itself cannot imply the uniform stability; see \cite{TB, BM, BTV, Koz, HMST, Dai-13} for counterexamples.
\textit{Under what conditions is a given system like $(\ref{eq1.1})$ stable?} Cf.~\cite[Problem~10.2]{BTT} and also \cite{Gur,SWMWK}. Except for elementary cases such as by the joint spectral radius of $\A$ being less than $1$ (cf.~\cite{Bar, Dai-LAA}), no satisfactory conditions are presently available for checking the stability of (\ref{eq1.1}); in fact the problem is open even in the case of matrices of dimension two (cf.~\cite[p.~305]{BTT}).

\subsubsection{Robust periodic stability of MJLS}
In the first part of this paper, we will study the absolute/uniform exponential stability of the MJLS $(\ref{eq1.1})$ under the following robustness condition:

\begin{defn}\label{def1.1}
The MJLS $(\ref{eq1.1})$ is called \textit{robustly periodically stable}, provided that one can find an $\varepsilon>0$ such that if the ordered set $\B=\{B_1,\dotsc,B_K\}\subset\mathbb{R}^{d\times d}$ satisfies $\pmb{\|}A_k-B_k\pmb{\|}<\varepsilon$ for each $1\le k\le K$, then for any $n\ge1$ and any word $w\in W_\textrm{per}^n(\mathbb{A})$ there $\rho(B_{w_n}\dotsm B_{w_1})<1$.
\end{defn}

By considering $B_k=(1+\varepsilon/2)A_k$ for $1\le k\le K$, it is easily seen that if the MJLS $(\ref{eq1.1})$ is robustly periodically stable, then it is \textit{completely periodically stable}; i.e.,
\begin{equation}\label{eq1.3}
\exists\,\bgamma<1\textrm{ such that }\rho(A_{w_n}\dotsm A_{w_1})\le\bgamma\;\forall w\in W_\textrm{per}^n(\mathbb{A})\textrm{ and }n\ge1.
\end{equation}
Indeed, we can gain more, see Lemma~\ref{lem1.3} below.

Then in the special but classical case where $(\ref{eq1.1})$ is free of any constraints (i.e. $\mathbb{A}_{ij}\equiv1$ for all $1\le i,j\le K$),
(\ref{eq1.1}) is robustly periodically stable if and only if (\ref{eq1.1}) is uniformly exponentially stable;
see, e.g., \cite{SWP, Dai-LAA}. However, in our present context with the nontrivial constraint $\mathbb{A}$, for this equivalence, there appears an essential obstruction for one could not employ Elsner's reduction theorem (\cite{Els, Dai-JMAA}) to guarantee the product boundedness, i.e. $\pmb{\|}A_{\sigma(m+n)}\dotsm A_{\sigma(m+1)}\pmb{\|}\le\beta$ uniformly for $n\ge1$ and $m\ge0$, for any $\sigma\in\varSigma_\mathbb{A}^+$, as done in \cite{Dai-LAA} for a case with additional condition; this is because the set of matrices
$\{A_{\sigma(m+n)}\dotsm A_{\sigma(m+1)}\,|\,\sigma\in\varSigma_\mathbb{A}^+, m\ge0, n\ge1\}$
is by no means a semigroup under the matrix multiplication in general for our situation. For counterexamples, see Theorem~\ref{thm4.1} proved in Section~\ref{sec4.2}.

However, motivated by Liao~\cite{Liao-CAM}, Aoki~\cite{Aok}, Hayashi~\cite{Hay-ETDS}, Gan and Wen~\cite{GW}, and Dai~\cite{Dai-SCM} in the theory of differentiable dynamical systems, we can obtain the following result even through without Elsner's reduction theorem.

\begin{theorem}\label{thm1.2}
For any $\A=\{A_1,\dotsc,A_K\}\subset\mathbb{R}^{d\times d}$ and any $K$-by-$K$ $\{0,1\}$-matrix $\mathbb{A}$, the MJLS $(\ref{eq1.1})$ is uniformly exponentially stable if and only if it is robustly periodically stable.
\end{theorem}

This positively answers the unsolved problem \cite[Problem~10.2]{BTT} from the viewpoint of perturbation theory.

\subsubsection{Outline of the proof of Theorem~\ref{thm1.2}}\label{sec1.1.2}
The fact that for arbitrary switching (discrete-time, without the $\mathbb{A}$-constraint) systems the uniform exponential stability is equivalent to the robust periodic stability is well known (and in fact it is equivalent to the continuity of the joint spectral radius of $\A$). However our new element is in considering systems steered by  switching laws which are performed in accordance to restrictions imposed by the matrix $\mathbb{A}$.

The ``only if'' part of Theorem~\ref{thm1.2} holds trivially. So, we in fact need only prove the ``if'' part. For that, we first obtain the following simple result, which is similar to \cite[Theorem~1]{Fra}.

\begin{lemma}\label{lem1.3}
If the MJLS $(\ref{eq1.1})$ is robustly periodically stable, then it is ``dilation'' periodically stable in the sense that there exists an $\varepsilon>0$ such that for any real number $\alpha\in[1,1+\varepsilon)$ and any $n\ge1$,
\begin{equation}\label{eq1.4}
\alpha^n\rho(A_{w_n}\dotsm A_{w_1})<1\quad \forall w\in W_{\mathrm{per}}^n(\mathbb{A})\textrm{ and }n\ge1.
\end{equation}
\end{lemma}

This dilation property~(\ref{eq1.4}) is obviously stronger than the complete periodic stability property $(\ref{eq1.3})$.
Then this together with the following so-called Gel'fand-Berger-Wang spectral formula of MJLS implies Theorem~\ref{thm1.2}.

\begin{theorem}[Gel'fand-Berger-Wang Formula of MJLS]\label{thm1.4}
For any finite subset $\A=\{A_1,\dotsc,A_K\}$ of $\mathbb{R}^{d\times d}$ and any $\mathbb{A}\in\{0,1\}^{K\times K}$, it holds that
\begin{equation*}
\brho(\A,\mathbb{A}):=\limsup_{n\to\infty}\max_{w\in W_{\mathrm{per}}^n(\mathbb{A})}\sqrt[n]{\rho(A_{w_n}\dotsm A_{w_1})}=\lim_{n\to\infty}\max_{\sigma\in \varSigma_{\mathbb{A}}^+}\sqrt[n]{\pmb{\|}A_{\sigma(n)}\dotsm A_{\sigma(1)}\pmb{\|}}.
\end{equation*}
\end{theorem}
\noindent Here $\brho(\A,\mathbb{A})$ is called the generalized/joint spectral radius of $\A$ restricted to $\mathbb{A}$; see Definition~\ref{def1.13} below for a more general case.

We note that the classical Berger-Wang formula is just the special case of $\mathbb{A}_{ij}\equiv 1$ (cf.~\cite{BW}), which was recently generalized to sets of precompact operators by Morris~\cite{Mor12} in the case free of any constraint. Wirth~\cite{Wir05} presented a continuous-time version obeying a kind of special constraint. Theorem~\ref{thm1.4} gives a positive answer to \cite[Question~1]{Dai-LAA} in the case of $\mathbb{A}$-constraint.

Now, if the MJLS (\ref{eq1.1}) is robustly periodically stable, then from Lemma~\ref{lem1.3} and Theorem~\ref{thm1.4} it follows that $\brho(\A,\mathbb{A})<1$. Further we can see that (\ref{eq1.1}) is uniformly exponentially stable by using \cite[Corollary~2.8]{Dai-LAA}.

So to prove Theorem~\ref{thm1.2}, it is sufficient to prove Theorem~\ref{thm1.4}. See Section~\ref{sec2R} for the details.

\subsubsection{The continuity of the joint spectral radius of MJLS}
C.~Heil and G.~Strang in \cite{HS} proved that the joint spectral radius of $\A$ free of any constraints is continuous with respect to $\A$ in the topological space $\mathrm{C}(\{1,\dotsc,K\},\mathbb{R}^{d\times d})$. In fact, Wirth~\cite{Wir-02} proved that it is Lipschitz continuous with respect to $\A$, and Kozyakin~\cite{Koz-LAA} explicitly computed the related Lipschitz constant.

As a simple consequence of our Gel'fand-Berger-Wang formula of MJLS, we can obtain the following result under the constraint $\mathbb{A}$.

\begin{cor}\label{cor1.5}
Given any $\mathbb{A}\in\{0,1\}^{K\times K}$, $\brho(\A,\mathbb{A})$ is continuous with respect to $\A$ in the topological space $\mathrm{C}(\{1,\dotsc,K\},\mathbb{R}^{d\times d})$.
\end{cor}

\begin{proof}
Note that for any word $w\in W_{\mathrm{per}}^n(\mathbb{A})$ where $n\ge1$,
there holds that
\begin{equation*}
w^k:=(\stackrel{k\textrm{-time}}{\overbrace{w,\dotsc,w}})\in
W_{\mathrm{per}}^{kn}(\mathbb{A})\quad \textrm{and}\quad
\sqrt[k]{\rho((A_{w_n}\dotsm A_{w_1})^k)}=\rho(A_{w_n}\dotsm A_{w_1}).
\end{equation*}
So
$\sqrt[kn]{\rho((A_{w_n}\dotsm A_{w_1})^k)}=\sqrt[n]{\rho(A_{w_n}\dotsm A_{w_1})}$ for
all $k\ge 1$, and it follows from Theorem~\ref{thm1.4} that
\begin{equation*}
\brho(\A,\mathbb{A})=\sup_{n\ge N}\left\{\max_{w\in
W_{\mathrm{per}}^n(\mathbb{A})}\sqrt[n]{\rho(A_{w_n}\dotsm A_{w_1})}\right\}\quad
\forall N\ge1.
\end{equation*}
Then the statement follows immediately from
\begin{equation*}
\sup_{n\ge1}\left\{\max_{w\in
W_{\mathrm{per}}^n(\mathbb{A})}\sqrt[n]{\rho(A_{w_n}\dotsm A_{w_1})}\right\}=\brho(\A,\mathbb{A})=\inf_{n\ge
1}\left\{\max_{\sigma\in\varSigma_{\mathbb{A}}^+}\sqrt[n]{\pmb{\|}A_{\sigma(n)}\dotsm A_{\sigma(1)}\pmb{\|}}\right\}.
\end{equation*}
This completes the proof of Corollary~\ref{cor1.5}.
\end{proof}

This shows that our Gel'fand-Berger-Wang formula of MJLS is of independent interest.
\subsection{Random linear ordinary differential equations}\label{sec1.2}
We now turn to the second part of this paper -- continuous-time random linear dynamical systems driven by a topological semiflow.
\subsubsection{The driving semiflow with closing by periodic orbits property}\label{sec1.2.1}
We first introduce our driving dynamical system for the random linear ordinary differential equation considered later. Let
\begin{equation*}
\varphi\colon\mathbb{R}_+\times W\rightarrow W;\quad (t,w)\mapsto t\bcdot{w},\quad \textrm{where }\mathbb{R}_+=[0,\infty),
\end{equation*}
be a continuous semiflow on a compact metric space $(W,\bd)$. A point $p\in W$ is called a periodic point of $\varphi$ of period $\tau>0$ if and only if $p=\tau\bcdot p$.
By $\mathcal{P}(\varphi)$ we denote the set of all periodic points of $\varphi$. It includes all the fixed points of $\varphi$.

We now introduce the following important condition -- closing by periodic orbits property -- for our driving system $\varphi$.

\begin{defn}\label{def1.6}
We say that the semiflow $\varphi\colon(t,w)\mapsto t\bcdot w$ has the \textit{closing by periodic orbits property}, provided that to any $\varepsilon>0$ there corresponds a constant $\delta>0$ such that if $w\in W$ satisfies $0<\bd(w,\tau{\bcdot}w)<\delta$ for some $\tau\ge1$, then one can pick up some point $\hat{w}\in W$ such that
$\hat{w}=\tau{\bcdot}\hat{w}$ and $\bd(t{\bcdot}w, t{\bcdot}\hat{w})<\varepsilon$ for all $0\le t\le \tau$.
\end{defn}

For the closing by periodic orbits property induced by the uniform hyperbolicity of a $\mathrm{C}^1$-vector field on manifolds, the period of $\hat{w}$ and its $\varepsilon$-shadowing property both permit a slight time-drift. In this case, our statement still holds by a slight modification of the proof presented here. Here our definition is convenient for our argument below.

By a standard argument, see e.g. \cite[Lemma~3.8]{Dai-FM}, we can see that the closing by periodic orbits property implies that for any ergodic probability measure $\mu$ of $\varphi$ on $W$, one can choose a sequence of periodic orbits $\{P_n\}_1^\infty$ of $\varphi$ such that
\begin{equation*}
\mu_{P_n}\xrightarrow{\textrm{in the weak-$*$ topology}}\mu\quad\textrm{and}\quad P_n\xrightarrow{\textrm{in the Hausdorff metric}}\mathrm{supp}(\mu)\quad \textrm{as }n\to+\infty,
\end{equation*}
where $\mu_{P}$ denotes the unique ergodic probability measure of $\varphi$ supported on the periodic orbit $P$. This point is important for the approximation of Lyapunov exponents by periodic orbits later in Sections~\ref{sec2R} and \ref{sec3}.

\subsubsection{Robustly periodically stable random linear ordinary differential equations}\label{sec1.2.2}
For any $\bX\in\mathrm{C}(W,\mathbb{R}^{d\times d})$, it naturally gives rise to a random linear differential dynamical system described as follows:
\begin{equation}\label{eq1.5}
\dot{x}(t)=\bX(t{\bcdot}w)x(t)\quad (t\in\mathbb{R}_+, x\in\mathbb{R}^d\textrm{ and }w\in W)
\end{equation}
driven by the semiflow $\varphi\colon(t,w)\mapsto t{\bcdot}w$ as in Section~\ref{sec1.2.1}.

By $\mathscr{X}(t,w)$ we mean the corresponding standard fundamental matrix solution of (\ref{eq1.5}); that is to say, $\mathscr{X}(0,w)=I_d$ the $d\times d$ identity matrix, and
$\frac{d}{dt}\mathscr{X}(t,w)=\bX(t{\bcdot}w)\mathscr{X}(t,w)$ for all $t>0$ and $w\in W$. Then for any initial value $x_0\in\mathbb{R}^d$ and each driving point $w\in W$, the corresponding solution of (\ref{eq1.5}) is $x_w(t,x_0)=\mathscr{X}(t,w)x_0$ for $t\in\mathbb{R}_+$. So,
\begin{equation*}
(\varphi,\mathscr{X})\colon\mathbb{R}_+\times W\times\mathbb{R}^d\rightarrow W\times\mathbb{R}^d;\quad (t,(w,x))\mapsto\left(t{\bcdot}w,\mathscr{X}(t,w)x\right)
\end{equation*}
is a continuous skew-product semiflow, and
\begin{equation*}
\mathscr{X}\colon\mathbb{R}_+\times W\rightarrow \mathrm{GL}(d,\mathbb{R});\quad (t,w)\mapsto\mathscr{X}(t,w)
\end{equation*}
forms a smooth linear cocycle driven by the continuous semiflow $\varphi\colon\mathbb{R}_+\times W\rightarrow W$.

Similar to the discrete-time case, $\bX$ is said to be \textit{absolutely exponentially stable} driven by $\varphi$, if for any driving point $w\in W$ and any initial value $x_0\in\mathbb{R}^d$, the solution $x(t)=x_w(t,x_0)$ of (\ref{eq1.5}) is such that
\begin{equation*}
\limsup_{t\to\infty}\frac{1}{t}\log\|x(t)\|<0;
\end{equation*}
or equivalently, one can find two constants $C>0$ and $0<\gamma<1$ so that
\begin{equation*}
\pmb{\|}\mathscr{X}(t,w)\pmb{\|}\le C\gamma^t\quad\forall t\ge0\textrm{ and }w\in W;
\end{equation*}
that is, $\bX$ is \textit{uniformly exponentially stable} driven by $\varphi$.

The following concept is the continuous-time version of Definition~\ref{def1.1}, which is motivated by the $\mathrm{C}^1$-weak-star property introduced by Dai~\cite{Dai-JMSJ} for $\mathrm{C}^1$-vector fields on closed manifolds.

\begin{defn}\label{def1.7}
$\bX\in\mathrm{C}(W,\mathbb{R}^{d\times d})$ is said to have the \textit{robust periodic stability} driven by $\varphi$, if one can find some $\varepsilon>0$ such that if $\bY\in\mathrm{C}(W,\mathbb{R}^{d\times d})$ satisfies $\pmb{\|}\bX-\bY\pmb{\|}<\varepsilon$ then for any periodic or fixed point $\tilde{w}$ of $\varphi$, the induced linear system $\dot{y}(t)=\bY(t\bcdot\tilde{w})y(t)$ is stable; that is, there are constants $K_{\bY,\tilde{w}}>0$ and $0<\gamma_{\bY,\tilde{w}}^{}<1$ so that $\pmb{\|}\mathscr{Y}(t,\tilde{w})\pmb{\|}\le K_{\bY,\tilde{w}}\gamma_{\bY,\tilde{w}}^t$ for all $t\ge0$.
\end{defn}

This concept is similar to, but weaker than, the classical $\mathrm{C}^1$-star
property for a $\mathrm{C}^1$-differentiable dynamical system that is defined on a closed manifold and has been deeply studied by \cite{Liao-CAM,GW,BJ,Dai-SCM} etc. in the continuous-time case and by \cite{Man,Aok,Hay-ETDS,Arb} and so on in the discrete-time case.

In the second part of this paper, we shall mainly prove the following stability criterion of continuous-time linear cocycles.

\begin{theorem}\label{thm1.8}
Let $\varphi\colon(t,w)\mapsto t{\bcdot}w$ be a continuous semiflow on the compact metric space $W$, which has the closing by periodic orbits property; and let $\bX\in\mathrm{C}(W,\mathbb{R}^{d\times d})$. Then driven by $\varphi$, $\bX$ is robustly periodically stable if and only if it is uniformly exponentially stable.
\end{theorem}

It should be noted here that without the robustness condition, only the periodic or even the complete periodic stability of $\bX$ does not need to imply the uniform stability in general, as shown by Theorem~\ref{thm1.14} below and Theorem~\ref{thm4.1} presented in Section~\ref{sec4.2}.

In addition, we should note that strictly speaking, Theorem~\ref{thm1.2} is not a discrete-time version of Theorem~\ref{thm1.8} because of the following reason: For the MJLS (\ref{eq1.1}), its driving system is the one-sided Markovian subshift transformation of finite-type:
\begin{equation}\label{eq1.6}
\theta_\mathbb{A}^+\colon\varSigma_\mathbb{A}^+\rightarrow\varSigma_\mathbb{A}^+;\quad\sigma(\bcdot)\mapsto\sigma(\bcdot+1)
\end{equation}
and the generator of the cocycle is
\begin{equation}
\A\colon\varSigma_\mathbb{A}^+\rightarrow\mathbb{R}^{d\times d}\quad \textrm{by }\sigma\mapsto A_{\sigma(1)}.
\end{equation}
So under the situation of Theorem~\ref{thm1.8} the admissible $\varepsilon$-perturbation $\B$ belongs to $\mathrm{C}(\varSigma_\mathbb{A}^+,\mathbb{R}^{d\times d})$, which is not necessarily to be locally constant and so it does not need to generate a MJLS of type (\ref{eq1.1}) with $\B$ instead of $\A$. However, under the situation of Theorem~\ref{thm1.2}, our admissible $\varepsilon$-perturbation $\B$ belongs to $\mathrm{C}(\{1,\dotsc,K\},\mathbb{R}^{d\times d})$ which can give rise to another MJLS $\varepsilon$-close to the MJLS (\ref{eq1.1}). This is just the reason why we need Theorem~\ref{thm1.4}.

Comparing with the $\mathrm{C}^1$-star condition in the differentiable systems, since here the driving system $\varphi\colon\mathbb{R}_+\times W\rightarrow W$ has been given previously and it is independent of the ``hyperbolicity'' of $\bX$, we cannot make use of the classical shadowing lemma of Liao for quasi-hyperbolic orbit arc to construct a self-contradictory periodic orbit, when we assume the statement of Theorem~\ref{thm1.8} would not be true. Moreover, for a $\mathrm{C}^1$-vector field, it has at most a finite number of stable periodic orbits. However, our driving semiflow in Theorem~\ref{thm1.8} may have infinitely many stable periodic orbits.

\subsubsection{Outline of the proof of Theorem~\ref{thm1.8}}
Under the robust periodic stability condition (Definition~\ref{def1.7}), similar to Lemma~\ref{lem1.3} we can easily observe the fact that there exists a constant $0<\bgamma<1$ such that
\begin{equation}
\sqrt[\tau]{\rho(\mathscr{X}(\tau,w))}\le\bgamma\quad\forall w\in\mathcal{P}_\tau(\varphi)\textrm{ and }\tau>0,
\end{equation}
where $\mathcal{P}_\tau(\varphi)$ denotes the set of all periodic points of $\varphi$ of period $\tau$. But crucially there exists no an analogous Gel'fand-Berger-Wang formula under the situation of Theorem~\ref{thm1.8}. (In fact, it is an open problem!)

Now differently from Theorem~\ref{thm1.4}, by using Liao's perturbation methods developed in \cite{Liao-AMS, Liao-CAM}, to prove Theorem~\ref{thm1.8} we will first prove the following

\begin{theorem}[Quasi contraction lemma]\label{thm1.9}
There are constants $\bta<0$ and $\bT>0$ such that,
if $w\in\mathcal {P}(\varphi)$ has the period
$\pi_{w}\ge\bT$ and if
$0=t_0<t_1<\cdots<t_\ell=\pi_w$, $\ell\ge 1$,
is a subdivision of the interval $[0,\pi_w]$ satisfying $t_k-t_{k-1}\ge\bT$ for
$k=1,\ldots,\ell$, then
\begin{equation*}
\frac{1}{\pi_w}\sum_{k=1}^\ell\log\pmb{\|}\mathscr{X}(t_k-t_{k-1},{t_{k-1}}{\bcdot}w)\pmb{\|}\le\bta.
\end{equation*}
\end{theorem}

This quasi contraction implies that $\rho(\mathscr{X}(\pi_w,w))<\exp(\pi_w\bta)$ for any $w\in\mathcal{P}(\varphi)$ of period $\pi_w$. So it is stronger than the dilation property (\ref{eq1.4}).

Based on this theorem and some semi-uniform ergodic arguments (Lemmas~\ref{lem3.3R} and \ref{lem3.4R}), we can prove Theorem~\ref{thm1.8}; see Section~\ref{sec3} for the details.

\subsection{Linear cocycle driven by an irrational rotation}\label{sec1.3}
To construct counterexamples to our robustness condition, we need to consider linear cocycles driven by an irrational rotation in the third part of this paper.

Let $\mathbb{T}^1$ be the unit circle in the complex plane $\mathbb{C}^1$, let $\mathscr{B}(\mathbb{T}^1)$ be the Borel $\sigma$-field generated by the open arcs, and let $\bnu$ be the normalized circular Lebesgue measure: map $[0,1)$ to the unit circle by $\phi(x)=e^{2\pi\mathfrak{i}x}$ and $\bnu(A)=\verb"Leb"(\phi^{-1}(A))$ for $A\in\mathscr{B}(\mathbb{T}^1)$, where $\mathfrak{i}^2=-1$ and \verb"Leb"($\cdot$) is the standard Lebesgue measure on $[0,1)$. For any fixed $x\in[0,1)$, we define
\begin{equation}\label{eq1.9}
R_x\colon\mathbb{T}^1\rightarrow\mathbb{T}^1;\quad w\mapsto we^{2\pi\mathfrak{i}x}.
\end{equation}
Since $R_x$ is effectively the rotation of the circle through the angle $2\pi x$ for each $x\in[0,1)$, $R_x$ preserves $\bnu$ and it is ergodic if and only if $x$ is irrational.

Recall that for an irrational number $\omega$, there is the following rational approximation theorem:

\begin{Dirichlet}
If $\omega\in[0,1)$ is irrational, then one can find a sequence of integer pairs $(p_n,q_n)\in\mathbb{N}\times \mathbb{N}$ such that $p_n/q_n$ are irreducible, $q_n\uparrow+\infty$, and
\begin{equation*}
\big{|}\omega-\frac{p_n}{q_n}\big{|}<\frac{1}{q_n^2}
\end{equation*}
as $n$ tends to $+\infty$. Here one can require $p_n/q_n\uparrow\omega$ (resp. $p_n/q_n\downarrow\omega$) as $n\to+\infty$.
\end{Dirichlet}

Next, for any irrational $\omega\in[0,1)$ and any continuous matrix-valued function
\begin{equation*}
\bS\colon\mathbb{T}^1\rightarrow\mathrm{GL}(d,\mathbb{R});\quad w\mapsto S(w),
\end{equation*}
we consider the stability of the induced linear cocycle
\begin{equation*}
\mathscr{S}_\omega\colon\mathbb{N}\times\mathbb{T}^1\rightarrow\mathrm{GL}(d,\mathbb{R});\quad (n,w)\mapsto S(R_\omega^{n-1}(w))\dotsm S(R_\omega(w))S(w)
\end{equation*}
driven by the irrational rotation $R_\omega\colon\mathbb{T}^1\rightarrow\mathbb{T}^1$, which is such that
\begin{equation*}
\mathscr{S}_\omega(m+n,w)=\mathscr{S}_\omega(n,R_\omega^m(w))\mathscr{S}_\omega(m,w)\quad \forall w\in\mathbb{T}^1\textrm{ and }m,n\in\mathbb{N},
\end{equation*}
under the following robustness condition.

\begin{defn}\label{def1.10}
For an irrational $\omega\in[0,1)$ and $\bS\in\mathrm{C}(\mathbb{T}^1,\mathrm{GL}(d,\mathbb{R}))$, the cocycle $\mathscr{S}_\omega$ is called \textit{robustly periodically stable}, provided that one can find an $\varepsilon>0$ and a sequence of integer pairs $(p_n,q_n)$ as in Dirichlet's theorem such that if $\B\in\mathrm{C}(\mathbb{T}^1,\mathrm{GL}(d,\mathbb{R}))$ satisfies $\pmb{\|}\bS-\B\pmb{\|}<\varepsilon$ then for any $n\ge1$ (sufficiently large),
\begin{equation*}
\rho\big{(}B(R_{x_n}^{q_n-1}(w))\dotsm B(R_{x_n}(w))B(w)\big{)}<1\quad\forall w\in\mathbb{T}^1.
\end{equation*}
Here $x_n=p_n/q_n$ for $n\ge1$.
\end{defn}

Then we can obtain the following uniform stability result.

\begin{theorem}\label{thm1.11}
Let $\omega\in(0,1)$ be irrational and $\bS\in\mathrm{C}(\mathbb{T}^1,\mathrm{GL}(d,\mathbb{R}))$. If the induced cocycle $\mathscr{S}_\omega$ is robustly periodically stable, then it is uniformly exponentially stable.
\end{theorem}

We will prove this theorem in Section~\ref{sec4.1} based on Theorems~\ref{thm1.9} and \ref{thm1.8}. In fact, it is a consequence of the above continuous-time version Theorem~\ref{thm1.8}.

\subsection{Spectral finiteness of linear cocycle}\label{sec1.4}

We can now construct examples which show that the robustness condition is sharp for the above Theorems~\ref{thm1.2}, \ref{thm1.8} and \ref{thm1.11}. Let
\begin{equation*}
T\colon W\rightarrow W
\end{equation*}
be a continuous transformation on the compact metric space $(W,\bd)$.

Now the discrete-time version of Definition~\ref{def1.6} can be stated as follows.

\begin{defn}\label{def1.12}
$T$ is said to have the \textit{closing by periodic orbits property}, provided that to any number $\varepsilon>0$ there corresponds a constant $\delta>0$ such that if $w\in W$ satisfies $\bd(w,T^N(w))<\delta$ for some $N>0$, then one can pick up some point $\hat{w}\in W$ with the property:
$\hat{w}=T^N(\hat{w})$ and $\bd(T^k(w), T^k(\hat{w}))<\varepsilon$ for $0\le k\le N$.
\end{defn}

It is a well-known fact that the finite-type subshift transformation (\ref{eq1.6}) has this closing by periodic orbits property (cf.~Lemma~\ref{lem2.5R}). In addition, if $T$ is a hyperbolic $\mathrm{C}^1$-diffeomorphism on a closed manifold $M^n$, then it has the closing by periodic orbits property \cite{Sma}; if $T$ is a nonuniformly hyperbolic $\mathrm{C}^{1+\alpha}$-diffeomorphism it has this property \cite{Kat}; moreover, if $T$ has the quasi-hyperbolic property, then it also has this property \cite{Liao79, Dai-10}.

Let
\begin{equation*}
\bC\colon W\rightarrow \mathbb{R}^{d\times d};\quad w\mapsto C(w)
\end{equation*}
be a continuous matrix-valued function. Based on $\bC$, there gives rise to the linear cocycle
\begin{equation*}
\mathscr{C}\colon\mathbb{N}\times W\rightarrow\mathbb{R}^{d\times d};\quad (n,w)\mapsto C(T^{n-1}(w))\dotsm C(w)
\end{equation*}
driven by $T\colon W\rightarrow W$.

Since $\mathscr{C}(m+n,w)=\mathscr{C}(n,T^m(w))\mathscr{C}(m,w)$ for any $w\in W$ and $m,n\ge0$, we have
\begin{equation*}\begin{split}
\max_{w\in W}\pmb{\|}\mathscr{C}(m+n,w)\pmb{\|}&\le\max_{w\in W}\pmb{\|}\mathscr{C}(n,T^m(w))\pmb{\|}\cdot\max_{w\in W}\pmb{\|}\mathscr{C}(m,w)\pmb{\|}\\
&\le\max_{w\in W}\pmb{\|}\mathscr{C}(n,w)\pmb{\|}\cdot\max_{w\in W}\pmb{\|}\mathscr{C}(m,w)\pmb{\|}.
\end{split}\end{equation*}
Then we can well introduce the following concept.

\begin{defn}[{\cite[Definition~1.2]{Dai-LAA}}]\label{def1.13}
The \textit{joint spectral radius} of $\bC$ driven by $T$ is defined as
\begin{equation*}
\brho(\bC,T)=\limsup_{n\to\infty}\max_{w\in W}\sqrt[n]{\pmb{\|}\mathscr{C}(n,w)\pmb{\|}}\quad\left(=\lim_{n\to\infty}\max_{w\in W}\sqrt[n]{\pmb{\|}\mathscr{C}(n,w)\pmb{\|}}=\inf_{n\ge1}\max_{w\in W}\sqrt[n]{\pmb{\|}\mathscr{C}(n,w)\pmb{\|}}\right).
\end{equation*}
\end{defn}

From \cite[Lemma~2.5]{Dai-LAA} we see that $\log\brho(\bC,T)\in\mathbb{R}\cup\{-\infty\}$ is just the maximal Lyapunov exponent of the linear cocycle $\mathscr{C}$ driven by $T\colon W\rightarrow W$.

Related to $\brho(\bC,T)$, the following question is very interesting for optimization control, wavelets, numerical computation of spectral radius, random matrices and so on.

\begin{conj}[\cite{PR, DL, Gur, LW}]
Let the driving system $T\colon W\rightarrow W$ have the closing by periodic orbits property. Then for any $\bC\in\mathrm{C}(W,\mathbb{R}^{d\times d})$, does there exist a periodic point $w^*$ of $T$ of period $n^*$, for some $n^*\ge1$, such that
\begin{equation*}
\brho(\bC,T)=\sqrt[n^*]{\rho(\mathscr{C}(n^*,w^*))}\,\textrm{?}
\end{equation*}
\end{conj}

Let us first consider a simple example: By $\mathbb{M}^{K\times K}$ we denote the space of all $K$-by-$K$ Markov transition probability matrices; if $\bC\colon w\mapsto C(w)\in\mathbb{M}^{K\times K}$ for each $w\in W$, then $\brho(\bC,T)=1$ and the spectral finiteness of $\mathscr{C}$ holds from the Perron-Frobenius theorem. For some other positive cases, see \cite{Dai-13} and references therein.

Unfortunately, even in the MJLS case that $(W,T)=(\varSigma_\mathbb{A}^+,\theta_\mathbb{A}^+)$ with $\mathbb{A}_{ij}\equiv1, K=2$ and $d=2$, \cite{BM,BTV,Koz} all had disproved this spectral finiteness by offering the existence of infinitely many counterexamples; moreover, an explicit expression for such a counterexample has been found in the recent work of Hare \textit{et al.}~\cite{HMST}.

However the construction of all the counterexamples mentioned above are very technical and complicated.
We next will present very simple explicit counterexamples in our linear cocycle situation.

\begin{theorem}\label{thm1.14}
Let $\omega\in(0,1)$ be an irrational number, $W=[0,\omega]\times\mathbb{T}^1$, and
\begin{equation*}
T\colon W\rightarrow W,\quad \textrm{defined by }(y,z)\mapsto(y,R_y(z))\;\forall (y,z)\in W.
\end{equation*}
Then $T$ has the closing by periodic orbits property. Letting $A=\left(\begin{smallmatrix}1&\dotsm&1\\\vdots&\ddots&\vdots\\0&\dotsm&1\end{smallmatrix}\right)\in\mathbb{R}^{d\times d}$, we set
\begin{equation*}
\bC\colon W\rightarrow\mathbb{R}^{d\times d};\quad (y,z)\mapsto C(y,z)=\frac{y}{\omega}A.
\end{equation*}
Then the induced linear cocycle $\mathscr{C}$ driven by $T$ does not have the spectral finiteness.
\end{theorem}

\begin{proof}
Clearly, $\brho(\bC,T)=1$ and for every periodic point $(p/q,z)\in W$ of $T$ of period $q\ge1$, where $p,q\in\mathbb{N}$ with $p<q$ and ${p}/{q}<\omega$, we have
\begin{equation*}
\sqrt[q]{\rho(\mathscr{C}(q,w))}=\frac{p}{q\omega}<1,\quad \textrm{where }w=(p/q,z).
\end{equation*}
From Dirichlet's theorem, it follows that $T\colon W\rightarrow W$ has the closing by periodic orbits property; see Section~\ref{sec4.1} for the details. This thus completes the proof of Theorem~\ref{thm1.14}.
\end{proof}

We note here that for the driving dynamical system $T$ in Theorem~\ref{thm1.14}, the periodic points are dense in $W$. This theorem shows that the periodic stability does not need to imply the uniform exponential stability. So the robustness condition in Theorems~\ref{thm1.2}, \ref{thm1.8} and \ref{thm1.11} is necessary.

In addition, in Section~\ref{sec4.2} we shall present another counterexample which is completely periodically stable.
\subsection{An open question}\label{sec1.5}
We will conclude this introductory section with an open question, which should be important to the mathematical theory of control systems.

Given any matrix $\mathbb{A}\in\{0,1\}^{K\times K}$, let
\begin{equation*}
\mathscr{M}_K^{d\times d}=\left\{\A=(A_1,\dotsc,A_K)\,|\,A_k\in\mathbb{R}^{d\times d}\textrm{ for }1\le k\le K\textrm{ and }\brho(\A,\mathbb{A})=1\right\},
\end{equation*}
which is endowed with the natural topology as a subspace of the $K$-fold product topological space
$$M(K,d\times d)=\stackrel{K\textrm{-fold}}{\overbrace{\mathbb{R}^{d\times d}\times\dotsm\times\mathbb{R}^{d\times d}}}.$$
From Corollary~\ref{cor1.5}, we can easily see that $\mathscr{M}_K^{d\times d}$ is a nowhere dense closed subset of $M(K,d\times d)$.

Our Theorem~\ref{thm1.2} seems to be harmonic with the following

\begin{conjecture}[Dense Spectral Finiteness]
The $\A\in\mathscr{M}_K^{d\times d}$, which has the spectral finiteness restricted to $\mathbb{A}$, is dense in the topological space $\mathscr{M}_K^{d\times d}$.
\end{conjecture}

\section{The Gel'fand-Berger-Wang formula of MJLS}\label{sec2R}
As pointed out in Section~\ref{sec1.1.2}, to prove Theorem~\ref{thm1.2}, we need only prove our Gel'fand-Berger-Wang formula of MJLS (Theorem~\ref{thm1.4}).
This section is devoted to proving Theorem~\ref{thm1.4}. In fact, we shall prove a more general approximation theorem of Lyapunov exponents for MJLS.

Let $\A=\{A_1,\dotsc,A_K\}\subset \mathbb{R}^{d\times d}$ and $\mathbb{A}=(\mathbb{A}_{ij})\in\{0,1\}^{K\times K}$ be arbitrarily given as in Theorem~\ref{thm1.2}. Let
\begin{equation*}
\varSigma_{\mathbb{A}}=\left\{\sigma=(\sigma(n))_{n\in\mathbb{Z}}\in\{1,\dotsc,K\}^{\mathbb{Z}}\,|\,\mathbb{A}_{\sigma(n)\sigma(n+1)}=1\
\forall n\in\mathbb{Z}\right\}
\end{equation*}
be the space of the bi-sided symbolic sequences endowed with the standard metric
\begin{equation}\label{eq2.1R}
d(\sigma,\sigma^\prime)=\sum_{n=-\infty}^{+\infty}K^{-|n|}|\sigma(n)-\sigma^\prime(n)|,\quad\forall \sigma,\sigma^\prime\in\varSigma_{\mathbb{A}}.
\end{equation}
Since $\mathbb{A}$ is naturally defined in accordance with the Markov transition probability matrix $\bP$ as in Section~\ref{sec1.1}, we may assume that
\begin{itemize}
\item each row of $\mathbb{A}$ has at least one entry $1$ and moreover every column of $\mathbb{A}$ has at least one entry $1$; otherwise we can select a subset $\{k_1,\dotsc,k_r\}\subset\{1,2,\dotsc,K\}$ such that the submatrix $\mathbb{A}^\prime=(\mathbb{A}_{k_ik_j})\in\{0,1\}^{r\times r}$ possesses this property.
\end{itemize}
Then $(\varSigma_{\mathbb{A}},d)$ is a nonempty compact metric space and the bi-sided Markovian subshift transformation on it
\begin{equation*}
\theta_{\mathbb{A}}\colon\varSigma_{\mathbb{A}}\rightarrow\varSigma_{\mathbb{A}};\quad\sigma(\cdot)\mapsto\sigma(\cdot+1)
\end{equation*}
is homeomorphic. By (\ref{eq2.1R}), there is an integer $\mathcal{K}>0$ so that
\begin{equation}\label{eq2.2R}
\sigma(n)=\sigma^\prime(n)\quad \textrm{for}\quad \ -100\le n\le100, \quad \textrm{whenever }d(\sigma,\sigma^\prime)\le\frac{1}{\mathcal{K}}.
\end{equation}
From now on we will simply write $\theta$ instead of $\theta_{\mathbb{A}}$, and driven by $\theta$ we
define the linear cocycle
\begin{equation}\label{eq2.3R}
\mathcal {A}\colon\mathbb{Z}_+\times\varSigma_{\mathbb{A}}\rightarrow\mathbb{R}^{d\times d};\quad (n,\sigma)\mapsto\begin{cases}I_d& \textrm{if }n=0,\\A_{\sigma(n-1)}\dotsm A_{\sigma(0)}& \textrm{if }n\ge1.\end{cases}
\end{equation}
Clearly (\ref{eq2.3R}) is uniformly exponentially stable if and only if so is the MJLS (\ref{eq1.1}).

\subsection{Approximation of Lyapunov exponents}\label{sec2.1R}
Recall that a probability measure $\mu$ on the
Borel measurable space $(\varSigma_{\mathbb{A}},\mathscr{B})$ is
said to be $\theta$-invariant, write as
$\mu\in\mathcal{M}(\theta)$ if $\mu=\mu\circ \theta^{-1}$; i.e.
$\mu(B)=\mu(\theta^{-1}(B))$ for all
$B\in\mathscr{B}$. An $\theta$-invariant probability
measure $\mu$ is called ergodic, write
as $\mu\in\mathcal{M}_{\textit{erg}}(\theta)$,
provided that for $B\in\mathscr{B}$ there $\mu\left(B\vartriangle
\theta^{-1}(B)\right)=0$ implies $\mu(B)=1$ or
$0$, where $\vartriangle$ means the symmetric difference of two sets.

Given any $\mu\in\mathcal{M}(\theta)$, according to the classical multiplicative ergodic theorem, for $\mu$-a.e.
$\sigma\in\varSigma_{\mathbb{A}}$, the limit
\begin{equation}\label{eq2.4R}
\blambda(\sigma):=\lim_{n\to\infty}\frac{1}{n}\log\pmb{\|}\mathcal{A}(n,\sigma)\pmb{\|}
\end{equation}
exists and is called the (maximal) Lyapunov exponent of $\mathcal{A}$
at the switching law $\sigma$. Moreover, if $\mu$ is
ergodic, then $\blambda(\sigma)\equiv\blambda(\mu)$
$\mu$-almost surely for some constant
$\blambda(\mu)$, called the (maximal) Lyapunov exponent of
$\mathcal{A}$ at $\mu$.

To important is the following approximation of Lyapunov exponents by
those of $\mathcal{A}$ at periodic switching laws.

\begin{theorem}\label{thm2.1R}
Let $\mu\in\mathcal{M}_{\textit{erg}}(\theta)$
be arbitrarily given. Then, there exists a sequence $\{\sigma_k\}_{k=1}^{\infty}$ of periodic
switching laws in $\varSigma_{\mathbb{A}}$ such that
$\blambda(\sigma_k)\to\blambda(\mu)$ as
$k\to\infty$; in other words,
\begin{equation*}
\lim_{k\to\infty}\sqrt[\uproot{2}{\pi(\sigma_k)}]{\rho(\mathcal{A}(\pi(\sigma_k),\sigma_k))}=\lim_{n\to\infty}\sqrt[n]{\pmb{\|}\mathcal{A}(n,\sigma)\pmb{\|}},
\end{equation*}
for $\mu$-a.e. $\sigma\in\varSigma_{\mathbb{A}}$, where $\pi(\sigma_k)$ denotes the period of $\sigma_k$.
\end{theorem}

We note that the generators $(\sigma_k(0),\dotsc,\sigma_k(\pi(\sigma_k)-1))$ of the periodic laws $\sigma_k$ are subwords of $\sigma$ from the proof presented below.
Comparing this theorem to \cite[Theorem~1]{Dai-FM} that was proved using the
Pesin theory under the assumption of $\A$ being nonsingular there, we here do not require $\A$ to be
nonsingular using different approaches.

\subsubsection{A finer multiplicative ergodic theorem}\label{sec2.1.1}
The classical multiplicative ergodic theorem of Oseledets, for example, see
\cite{Ose}, asserts that for any given $\mu\in\mathcal{M}_{\textit{erg}}(\theta)$,
there are numbers
\begin{equation}\label{eq2.5R}
-\infty\le\lambda_\kappa(\mu)<\cdots<\lambda_1(\mu)=\blambda(\mu)\quad(1\le\kappa\le d)
\end{equation}
and a Borel measurable forwardly invariant filtration
\begin{equation}\label{eq2.6R}
\{0\}=V_{\kappa+1}(\sigma)\subset
V_\kappa(\sigma)\subset\cdots\subset
V_1(\sigma)=\mathbb{R}^d\quad \mu\textrm{-a.e. } \sigma\in\varSigma_{\mathbb{A}}
\end{equation}
such that for $1\le j\le\kappa$
\begin{equation}
\lambda_j(\mu)=\lim_{n\to\infty}\frac{1}{n}\log\|\mathcal{A}(n,\sigma)x_0\|\quad\forall x_0\in
V_j(\sigma)\setminus V_{j+1}(\sigma).
\end{equation}
Let $\mathfrak{m}_j(\mu)=\dim V_j(\sigma)-\dim V_{j+1}(\sigma)$ for $\mu$-a.e. $\sigma\in\varSigma_{\mathbb{A}}$, called the multiplicity of $\lambda_j(\mu)$. Then,
$$\Sigma_{\mathrm{Lya}}(\mathcal{A},\mu):=\{(\lambda_j(\mu), \mathfrak{m}_j(\mu))\,|\,j=1,\ldots,\kappa)\}$$
is called the Lyapunov spectrum of $\mathcal{A}$ at $\mu$.
It is convenient to rewrite $\Sigma_{\mathrm{Lya}}(\mathcal{A},\mu)$ as
\begin{equation*}
(\lambda_\kappa(\mu)=~)\quad\chi_d^{}(\mu)\le\chi_{d-1}^{}(\mu)\le\cdots\le\chi_1^{}(\mu)\quad(~=\lambda_1(\mu))
\end{equation*}
counting with multiplicities. For the case that $\mu$ is supported on a periodic switching law $\sigma$, we write its spectrum as
\begin{equation*}
\chi_d^{}(\sigma)^{}\le\chi_{d-1}^{}(\sigma)\le\cdots\le\chi_1^{}(\sigma).
\end{equation*}
The case that $\kappa=1$ is trivial for our arguments later. So,
we will mainly deal with $\kappa\ge2$.

From Theorem~\ref{thm2.1R} follows easily the more general approximation theorem.

\begin{cor}\label{cor2.2R}
Let $\mu\in\mathcal
{M}_{\textit{erg}}(\theta)$
be arbitrarily given. Then, there exists a sequence of periodic
switching laws in $\varSigma_{\mathbb{A}}$,
$\{\sigma_k\}_{k=1}^{\infty}$, such that for $1\le j\le d$,
\begin{equation*}
\chi_j^{}(\sigma_k)\to\chi_j^{}(\mu)\quad \textrm{as }
k\to\infty.
\end{equation*}
\end{cor}

\begin{proof}
For any $1\le\ell\le d$, we denote by $\bigwedge^\ell\mathbb{R}^d$ the $\ell$-th exterior power of $\mathbb{R}^d$. If $L\colon\mathbb{R}^d\rightarrow\mathbb{R}^d$
is a linear map, then it induces naturally a linear map $\bigwedge^\ell L\colon\bigwedge^\ell\mathbb{R}^d\rightarrow\bigwedge^\ell\mathbb{R}^d$. Define
\begin{equation*}
{\bigwedge}^\ell\A=\left\{{\bigwedge}^\ell A_k\right\}_{1\le k\le K},
\end{equation*}
where every matrix $A_k\in\mathbb{R}^{d\times d}$ is identified with the linear map $A_k\colon x\mapsto A_kx$. Let
\begin{equation*}
\Lambda_\ell(\mu)=\chi_1^{}(\mu)+\cdots+\chi_\ell^{}(\mu).
\end{equation*}
From the multiplicative ergodic theorem follows that $\Lambda_\ell(\mu)$ is just the maximal Lyapunov of exponent
of the cocycle induced by $\bigwedge^\ell\A$ at $\mu$ which is also driven by $\theta\colon\varSigma_{\mathbb{A}}\rightarrow\varSigma_{\mathbb{A}}$. So, the statement follows immediately from Theorem~\ref{thm2.1R}.
\end{proof}

To prove Theorem~\ref{thm2.1R}, we need the finer multiplicative ergodic theorem, due to Froyland, Lloyd and Quas, which may be stated, in
a special case, as follows:

\begin{lemma}[\cite{FLQ}]\label{lem2.3R}
For any $\mu\in\mathcal{M}_{\textit{erg}}(\theta)$, one can find a Borel set
$\varGamma\subset\varSigma_{\mathbb{A}}$ with
$\mu(\varGamma)=1$ and $\theta(\varGamma)=\varGamma$,
for which there is a Borel measurable splitting of $\mathbb{R}^d$
into subspaces
\begin{equation*}
\mathbb{R}^d=\mathbb{E}_1(\sigma)\oplus\cdots\oplus\mathbb{E}_\kappa(\sigma)\quad\forall \sigma\in\varGamma
\end{equation*}
such that
\begin{equation*}
\mathcal{A}(n,\sigma)\mathbb{E}_j(\sigma)\subseteq\mathbb{E}_j(\sigma(\bcdot+n))\quad\forall n\ge1
\end{equation*}
and
\begin{equation*}
\lambda_j(\mu)=\lim_{n\to\infty}\frac{1}{n}\log\|\mathcal{A}(n,\sigma)x_0\|\quad\forall x_0\in
\mathbb{E}_j(\sigma)\setminus\{0\}
\end{equation*}
for all $1\le j\le\kappa$.

\medskip
\noindent\textbf{Notes:} Moreover, there hold the following two
properties.
\begin{enumerate}
\item[$\mathrm{(1)}$] $V_j(\sigma)=\bigoplus_{r\ge
j}\mathbb{E}_r(\sigma)$ for all $1\le j\le\kappa$ and so
$\mathbb{R}^d=\mathbb{E}_1(\sigma)\oplus V_2(\sigma)$;

\item[$\mathrm{(2)}$] the convergence is uniform for $x_0$ restricted to the unit sphere of
$\mathbb{E}_j(\sigma)$.
\end{enumerate}
\end{lemma}

Then, based on this fined multiplicative ergodic theorem, one can
easily obtain the following.

\begin{lemma}\label{lem2.4R}
For any $\mu\in\mathcal{M}_{\textit{erg}}(\theta)$,
one can find an invariant Borel subset
$\varLambda_{\mu}\subset\varSigma_{\!\mathbb{A}}$ of $\mu$-measure $1$,
such that to any $\sigma\in\varLambda_{\mu}$ there
corresponds a sequence of positive integers
$n_k\nearrow\infty$ so that
\begin{enumerate}
\item[$\mathrm{(1)}$] $\sigma(\cdot+n_k)\to \sigma$ as
$k\to\infty$;

\item[$\mathrm{(2)}$]
$\mathbb{E}_1(\sigma(\cdot+n_k))\to\mathbb{E}_1(\sigma)$
and $V_2(\sigma(\cdot+n_k))\to
V_2(\sigma)$ as $k\to\infty$.
\end{enumerate}
\end{lemma}

\begin{proof}
Since the splitting
$\mathbb{R}^d=\mathbb{E}_1(\sigma)\oplus\cdots\oplus\mathbb{E}_\kappa(\sigma)$
and the filtration $V_\kappa(\sigma)\subset\cdots\subset
V_1(\sigma)$ over $\varGamma$ given by Lemma~\ref{lem2.3R} both are measurable
with respect to $\sigma\in\varGamma$, the statement
follows easily from the Poincar\'{e} recurrence theorem and the
Lusin theorem.
\end{proof}

This set $\varLambda_\mu$ is similar to the Pesin set in the
smooth ergodic theory; see, for example, \cite[Lemma~2.4]{Dai-FM}.

\subsubsection{Closing by periodic switching laws property}\label{sec2.1.2}%
Let $\varLambda_\mu$ be defined as in the Lemma~\ref{lem2.4R}.
For any switching law $\sigma\in\varLambda_\mu$,
Lemma~\ref{lem2.4R}-(1) above indicates that
$\sigma$ is a Poincar\'{e} recurrent point of the subshift dynamical
system
$\theta\colon\varSigma_{\mathbb{A}}\rightarrow\varSigma_{\mathbb{A}}$. Then it can be closed up by periodic switching laws from the following classical result.

\begin{lemma}\label{lem2.5R}
For any $\varepsilon>0$ with $\varepsilon<1/\mathcal{K}$, there exists a constant
$\delta=\delta(\varepsilon)>0$ and an integer $N=N(\varepsilon)\ge1$
such that if a switching law
$\sigma\in\varSigma_{\mathbb{A}}$
satisfies
$d(\sigma,\sigma(\cdot+\pi)<\delta$
for some $\pi\ge N$, then the periodic switching law of period
$\pi$ with generator $(\sigma(0),\dotsc,\sigma(\pi-1))$
\begin{equation*}
\eta=(\dotsc,\uwave{\sigma(0),\dotsc,\sigma(\pi-1)},\uwave{\overset{*}{\sigma}(0),\dotsc,\sigma(\pi-1)},\uwave{\sigma(0),\dotsc,\sigma(\pi-1)},\dotsc)
\end{equation*}
belongs to $\varSigma_{\!\mathbb{A}}$ such that
\begin{equation*}d(\sigma(\cdot+n),\eta(\cdot+n))\le\varepsilon\end{equation*}
for all $0\le n\le\pi-1$.
\end{lemma}

Here $\mathcal{K}$ is as in (\ref{eq2.2R}). This statement follows easily from (\ref{eq2.2R}) and so we omit its proof here.
\subsubsection{Invariant cones}\label{sec2.1.3}
Let $\varLambda_\mu$ be defined as in Lemma~\ref{lem2.4R}. For
any $\sigma\in\varLambda_\mu$, let
$$
P_\sigma\colon
\mathbb{R}^d\rightarrow\mathbb{E}_1(\sigma) \quad
\textrm{and}\quad Q_{\sigma}\colon \mathbb{R}^d\rightarrow
V_2(\sigma)
$$
be the natural projections satisfying
$Q_{\sigma}=\mathrm{Id}_{\mathbb{R}^d}-P_{\sigma}$. Since
$\mathbb{R}^d=\mathbb{E}_1(\sigma)\oplus V_2(\sigma)$,
both $P_{\sigma}$ and $Q_{\sigma}$ are well defined.

For any $\delta>0$ and any $\sigma\in\varLambda_\mu$,
we define a closed convex cone in $\mathbb{R}^d$ in the following
standard way:
\begin{equation}
\bK(\sigma,\delta)=\left\{x\in\mathbb{R}^d\colon
\|P_{\sigma}(x)\|\ge\delta^{-1}\|Q_{\sigma}(x)\|\right\}.
\end{equation}
Then, $\bK(\sigma,\delta)\to\mathbb{E}_1(\sigma)$ as
$\delta\to0$ in the sense of Grassmannian metric.

Borrowing the idea of the proof of \cite[Theorem~1.5]{Mor12}, we now estimate the ``distortion" of the cones
$\bK(\sigma,\delta)$ by the action of $\mathcal{A}(n,\sigma)$.

\begin{lemma}\label{lem2.6R}
Let $\epsilon>0$ be sufficiently small such that
$3\epsilon<\lambda_1(\mu)-\lambda_2(\mu)$ and
$\sigma\in\varLambda_\mu$ with an associated integer
sequence $n_k\nearrow\infty$ as in Lemma~\ref{lem2.4R}. Then,
there holds that
\begin{equation*}
\inf\limits_{x\in\bK(\sigma,1), \|x\|=1}\|\mathcal{A}(n_k,\sigma)x\|\ge\exp\left(n_k(\lambda_1(\mu)-3\epsilon)\right)
\end{equation*}
and
\begin{equation*}
\mathcal{A}(n_k, \sigma)\bK(\sigma,1)\subseteq
\bK(\sigma,1)
\end{equation*}
for all $k\ge1$ sufficiently large.
\end{lemma}

\begin{proof}
From Lemma~\ref{lem2.3R}, there follows that
\begin{equation*}
\lim_{n\to\infty}\frac{1}{n}\log\pmb{\|}\mathcal{A}(n,\sigma)|\mathbb{E}_1(\sigma)\pmb{\|}_{\min}=\lambda_1(\mu)
\quad\textrm{and}\quad
\lim_{n\to\infty}\frac{1}{n}\log\pmb{\|}\mathcal{A}(n,\sigma)|V_2(\sigma)\pmb{\|}=\lambda_2(\mu),
\end{equation*}
where $\pmb{\|}\cdot\pmb{\|}_{\min}$ denotes the minimum norm of a matrix or linear operator, which is given by $\pmb{\|}A\pmb{\|}_{\min}=\min_{x\in\mathbb{R}^{d}, \|x\|=1}\|Ax\|$ for any $A\in\mathbb{R}^{d\times d}$.

We now choose arbitrarily $\epsilon>0$ so small that
$3\epsilon<\lambda_1(\mu)-\lambda_2(\mu)$. If $n$ is large sufficiently, we
could obtain that for each $x\in \bK(\sigma,1)$,
\begin{equation}\label{eq2.9R}
\begin{split}
\|P_{\sigma(\cdot+n)}(\mathcal{A}(n,\sigma)x)\|&=\|\mathcal{A}(n,\sigma)P_{\sigma}(x)\|
\ge e^{n(\lambda_1(\mu)-\epsilon)}\|P_{\sigma}(x)\|\\
&\ge\frac{1}{2}e^{n(\lambda_1(\mu)-\epsilon)}\|x\|\\
&\ge e^{n(\lambda_1(\mu)-2\epsilon)}\|x\|
\end{split}
\end{equation}
and
\begin{equation}\label{eq2.10R}
\begin{split}
\|Q_{\sigma(\cdot+n)}(\mathcal{A}(n,\sigma)x)\|&=\|\mathcal{A}(n,\sigma)Q_{\sigma}(x)\|\\
&\le e^{n(\lambda_2(\mu)+\epsilon)}\|Q_{\mu}(x)\|\\
&\le e^{n(\lambda_2(\mu)+\epsilon)}\pmb{\|}Q_{\mu}\pmb{\|}\cdot\|x\|,
\end{split}
\end{equation}
where we have used the inequality
$$
\|x\|=\|P_{\sigma}(x)+Q_{\sigma}(x)\|\le
2\|P_{\sigma}(x)\|\quad\forall x\in \bK(\sigma,1).
$$
So, combining (\ref{eq2.9R}) and (\ref{eq2.10R}) leads to
$$
\|Q_{\sigma(\cdot+n)}(\mathcal{A}(n,\sigma)x)\|\le
e^{n(\lambda_2(\mu)+3\epsilon-\lambda_1(\mu))}\pmb{\|}Q_{\sigma}\pmb{\|}\cdot\|P_{\sigma(\cdot+n)}(\mathcal{A}(n,\sigma)x)\|
$$
for any $x\in \bK(\sigma,1)$, from which we could obtain that
for any $\delta>0$,
\begin{equation}\label{eq2.11R}
\mathcal{A}(n,\sigma)\bK(\sigma,1)\subseteq
\bK(\sigma(\cdot+n),\delta)
\end{equation}
when $n$ is large sufficiently. Moreover, from (\ref{eq2.9R}) and (\ref{eq2.10R}), we could obtain that
for any vector $x\in\bK(\sigma,1)$,
\begin{equation}\label{eq2.12R}
\begin{split}
\|\mathcal{A}(n,\sigma)x\|&\ge\|P_{\sigma(\cdot+n)}(\mathcal{A}(n,\sigma)x)\|-\|Q_{\sigma(\cdot+n)}(\mathcal{A}(n,\sigma)x)\|\\
&\ge\left(e^{n(\lambda_1(\mu)-2\epsilon)}-e^{n(\lambda_2(\mu)+\epsilon)}\pmb{\|}Q_{\sigma}\pmb{\|}\right)\|x\|\\
&\ge e^{n(\lambda_1(\mu)-2\epsilon)}\left(1-e^{n(\lambda_2(\mu)+3\epsilon-\lambda_1(\mu))}\pmb{\|}Q_{\sigma}\pmb{\|}\right)\|x\|
\end{split}
\end{equation}
which gives the first part of the statement.

To complete the proof of Lemma~\ref{lem2.6R}, by (\ref{eq2.9R}) together with
(\ref{eq2.10R}) we need to use only the property (2) of
Lemma~\ref{lem2.4R}.

Therefore, the proof of Lemma~\ref{lem2.6R} is completed.
\end{proof}
\subsubsection{Proof of Theorem~\ref{thm2.1R}}
Now we can prove Theorem~\ref{thm2.1R} using Lemmas~\ref{lem2.4R},
\ref{lem2.5R} and \ref{lem2.6R}.

\begin{proof}[Proof of Theorem~\ref{thm2.1R}]
Let $\mu\in\mathcal
{M}_{\textit{erg}}(\theta)$
be arbitrarily given. Without loss of generality, we assume
$\mu$ is non-atomic; otherwise, the statement holds
trivially.

The case that $\kappa=1$, i.e. the cocycle $\mathcal{A}$ has only one Lyapunov
exponent at $\mu$, is trivial from \cite{Dai09}. Now, assume $\kappa\ge2$ in Lemma~\ref{lem2.3R}.

Let $\epsilon>0$ be sufficiently small. Take arbitrarily
$\sigma=(\sigma(n))_{n\in\mathbb{Z}}\in\varLambda_\mu$ with an
associated increasing integer sequence $n_k\to\infty$ as in
Lemma~\ref{lem2.4R}. Write
\begin{gather*}
\omega_k=(\sigma(0),\dotsc,\sigma(n_k-1))\in\{1,\dotsc,K\}^{n_k}\\
\intertext{and}
\sigma_k=(\dotsc,\omega_k,\omega_k,\dotsc)\in\{1,\dotsc,K\}^\mathbb{Z}\quad\textrm{ with }\sigma_k(0)=\sigma(0).
\end{gather*}
From Lemmas~\ref{lem2.4R}-(1) and \ref{lem2.5R}, it
follows that $\omega_k\in
W_{\mathrm{per}}^{n_k}(\mathbb{A})$ and so
$\sigma_k\in\varSigma_{\mathbb{A}}$ as $k$ large enough.
Furthermore, if $k$ is large enough then from (\ref{eq1.2}) and Lemma~\ref{lem2.6R} there
follows that
\begin{align*}
\frac{1}{n_k}\log\rho(A_{\sigma(n_k-1)}\dotsm A_{\sigma(0)})&=\lim_{\ell\to\infty}\frac{1}{\ell n_k}\log\|\left(A_{\sigma(n_k-1)}\dotsm A_{\sigma(0)}\right)^\ell\|\\
&\ge\liminf_{\ell\to\infty}\frac{1}{\ell n_k}\log\|\left(A_{\sigma(n_k-1)}\dotsm A_{\sigma(0)}\right)^\ell x\|\\
&\ge \lambda_1(\mu)-3\epsilon
\end{align*}
for any $x\in \bK(\sigma,1)$ with $\|x\|=1$. This implies that
\begin{equation}\label{eq2.13R}
\blambda(\sigma_k)\ge\blambda(\mu)-3\epsilon\quad \textrm{as }k\textrm{ large sufficiently},
\end{equation}
and
\begin{equation}\label{eq2.14R}
\sup_{n\ge1}\max_{w\in W_{\mathrm{per}}^n(\mathbb{A})}\rho(A_{w_n}\dotsm A_{w_1})\ge e^{\blambda(\mu)-3\epsilon}.
\end{equation}
Without loss of generality, we may assume, by the
Krylov-Bogolioubov theory~\cite[Chaper~6]{NS}, that $\sigma$
is a quasi-regular point of $\mu$; that is to say,
\begin{equation*}
\lim_{n\to\infty}\frac{1}{n}\sum_{\ell=0}^{n-1}\varphi(\sigma(\cdot+\ell))=\int_{\varSigma_{\mathbb{A}}}\varphi\,d\mu\quad\forall\varphi\in
\mathrm{C}(\varSigma_{\mathbb{A}},\mathbb{R}).
\end{equation*}
For the periodic switching law
$\sigma_k\in\varSigma_{\mathbb{A}}$ of period
$n_k$, there exists uniquely an $\theta$-ergodic probability
measure, write $\mu_{\sigma_k}$, such that
$$
\mu_{\sigma_k}(\{y\})=\frac{1}{n_k}\quad\forall y\in\left\{\sigma_k(\cdot+i)\,|\,0\le i\le n_k-1\right\}.
$$
By the approximation theorem of ergodic
measures~\cite[Lemma~3.8]{Dai-FM}, we can assume that
$\mu_{\sigma_k}\to\mu$ as $k\to\infty$ in
the sense of weak-$*$ topology, choosing a subsequence if necessary.
Then from the upper semicontinuity of the maximal Lyapunov exponents
with respect to ergodic measures (c.f.~\cite{Dai09}), it follows
that
\begin{equation}\label{eq2.15R}
\blambda(\sigma_k)\le\blambda(\mu)+3\epsilon
\end{equation}
as $k$ is large sufficiently.

Combining (\ref{eq2.13R}) and (\ref{eq2.15R}) yields the desired
statement, since $\epsilon$ is arbitrary.

This thus proves Theorem~\ref{thm2.1R}.
\end{proof}
\subsection{Proof of Theorems~\ref{thm1.4} and \ref{thm1.2}}
\begin{proof}[Proof of Theorems~\ref{thm1.4}]
By Gel'fand's formula (\ref{eq1.2}) it is obvious that
\begin{equation*}
\limsup_{n\to\infty}\max_{w\in W_{\mathrm{per}}^n(\mathbb{A})}\sqrt[n]{\rho(A_{w_n}\dotsm A_{w_1})}\le\lim_{n\to\infty}\max_{\sigma\in \varSigma_{\mathbb{A}}^+}\sqrt[n]{\pmb{\|}A_{\sigma(n)}\dotsm A_{\sigma(1)}\pmb{\|}}=\lim_{n\to\infty}\max_{\sigma\in \varSigma_{\mathbb{A}}}\sqrt[n]{\pmb{\|}A_{\sigma(n-1)}\dotsm A_{\sigma(0)}\pmb{\|}}.
\end{equation*}
From \cite[Corollary~2.7]{Dai-LAA} it follows that there exists some $\mu\in\mathcal{M}_{\textit{erg}}(\theta)$ such that
\begin{equation*}
e^{\blambda(\mu)}=\lim_{n\to\infty}\max_{\sigma\in \varSigma_{\mathbb{A}}}\sqrt[n]{\pmb{\|}A_{\sigma(n-1)}\dotsm A_{\sigma(0)}\pmb{\|}}.
\end{equation*}
On the other hand, by Theorem~\ref{thm2.1R} (or (\ref{eq2.14R})) we can gain that
\begin{equation*}
\sup_{n\ge N}\max_{w\in W_{\mathrm{per}}^n(\mathbb{A})}\rho(A_{w_n}\dotsm A_{w_1})=\sup_{n\ge1}\max_{w\in W_{\mathrm{per}}^n(\mathbb{A})}\rho(A_{w_n}\dotsm A_{w_1})\ge e^{\blambda(\mu)}.
\end{equation*}
Therefore, the Gel'fand-Berger-Wang formula of MJLS holds in the situation of Theorem~\ref{thm1.4}.
\end{proof}

This therefore completes the proof of Theorem~\ref{thm1.2} as well.
\section{Robust periodic stability implies uniform stability for random linear ODE}\label{sec3}
This section will be devoted to proving Theorem~\ref{thm1.8} stated in Section~\ref{sec1.2} for random linear ODE that is driven by a topological semiflow with closing by periodic orbits property.
We shall first prove that under the robust periodic stability condition, there holds the quasi-contraction property described in Theorem~\ref{thm1.9}.
\subsection{A perturbation lemma of Liao}
Now we will introduce a perturbation lemma due to
Liao~\cite{Liao-AMS}. Let us consider a linear differential equations of order $d$
\begin{equation*}
\frac{dy}{dt}=a(t)y,\quad(t,y)\in\mathbb{R}_+\times\mathbb{R}^d,
\end{equation*}
where the $d$-by-$d$ coefficient matrix $a(t)$
is continuous in $t\in\mathbb{R}_+$ such that
\begin{equation*}
{\sup}_{t\in\mathbb{R}_+}\pmb{\|}a(t)\pmb{\|}\le \ba_*<\infty.
\end{equation*}
By $y_a(t,y_0)$ we mean its solution with $y_a(0,y_0)=y_0$ for any
initial value $y_0\in\mathbb{R}^d$. Let $\Phi_a(t)$ be the standard fundamental matrix solution of the above linear ODE with $\Phi_a(0)=I_d$ and set
\begin{equation*}
\Phi_a(t,s)=\Phi_a(t)\circ\Phi_a^{-1}(s),\quad \forall s,t\ge0.
\end{equation*}

We will need the following fundamental perturbation lemma proved by S.-T.~Liao in~\cite{Liao-AMS}.

\begin{lemma}[{\cite{Liao-AMS} also \cite[Lemma~3.2]{Dai-JMSJ}}]\label{lem3.1R}
Let $\varrho\in(0,1)$ and $0=t_0<t_1<\cdots<t_\ell=T<\infty$ and
$y_0,y_\natural\in\mathbb{R}^d$ such that
$\|y_0\|=1=\|y_a(T,y_\natural)\|$. If
\begin{equation*}
\|y_a(t_{\ell-1},y_\natural)\|=\pmb{\|}\Phi_a(t_{\ell-1},T)\pmb{\|}_{\min}
\end{equation*}
and
\begin{equation*}
t_k-t_{k-1}\ge\max\left\{\frac{16\ba_*\bar{T}}{\varrho},\, 2\lambda
\bar{T}+\frac{64}{\varrho}\log\frac{2}{\lambda_*},\,
\bar{T}+2\right\},\quad k=1,\ldots,\ell,
\end{equation*}
where
\begin{equation*}
\lambda=\frac{\varrho}{4\exp(2\ba_*)},\
\lambda_*=\frac{\lambda}{2}\exp(-\varrho/2)\quad \textrm{and}\quad
\bar{T}=\frac{32}{\lambda\varrho}\log\frac{32}{\lambda_*^2},
\end{equation*}
then there is a linear perturbed equation
\begin{equation*}
\frac{dy}{dt}=\left[a(t)+B_\natural(t)\right]y\qquad(t,y)\in\mathbb{R}_+\times\mathbb{R}^d
\end{equation*}
which satisfies the following two properties.
\begin{enumerate}
\item[$\mathrm{(i)}$] $B_\natural(t)$ is continuously differentiable in $t$
such that
\begin{equation*}
B_\natural(t)|_{\{0\}\cup[T-\frac{1}{8},\,\infty)}\equiv0\quad \textrm{and}\quad
{\sup}_{t\in\mathbb{R}_+}\pmb{\|}B_\natural(t)\pmb{\|}<\varrho.
\end{equation*}

\item[$\mathrm{(ii)}$] There exists a solution ${y}_\natural(t)$ to the perturbed equation such that ${y}_\natural(0)=y_0$ and
\begin{align*}
\frac{y_\natural(T)}{\|y_\natural(T)\|}&=y_a(T,{y}_\natural)\;\textrm{
or }=-y_a(T,{y}_\natural)\\
\intertext{and}\|y_\natural(t_k)\|&=\|y_\natural(t_{k-1})\|\cdot\pmb{\|}\Phi_a(t_k,t_{k-1})\pmb{\|}\quad \textrm{for }k=1,\ldots,\ell.\\
\intertext{Particularly,}
\|y_\natural(T)\|&=\prod_{k=1}^\ell\pmb{\|}\Phi_a(t_k,t_{k-1})\pmb{\|}.
\end{align*}
\end{enumerate}
\end{lemma}

In Lemma~\ref{lem3.1R}, the key point is that $\|y_\natural(T)\|=\prod_{k=1}^\ell\pmb{\|}\Phi_a(t_k,t_{k-1})\pmb{\|}$ which will implies the quasi contraction property in the following proof of Theorem~\ref{thm1.9}.

\subsection{Nonuniform stability on periodic orbits}\label{sec3.2}
Let $\varphi\colon\mathbb{R}_+\times W\rightarrow W;\, (t,w)\mapsto t{\bcdot}w$ be a continuous semiflow as in Section~\ref{sec1.2.1} of the Introduction. Recall that a point $w\in W$ is called a periodic point of $\varphi$ of period $\pi$ where $\pi>0$, if $w=\pi{\bcdot}w$. For a periodic point $w$, $\inf\{\pi>0;w=\pi{\bcdot}w\}$ is called the \textit{prime period} of $w$. Clearly, a periodic point $w$ is fixed (i.e. $w=t{\bcdot}w$ for all $t>0$) by $\varphi$ if and only if its prime period is $0$. By $\mathcal{P}(\varphi)$ we mean the set of all the periodic points of $\varphi$ including all fixed points of $\varphi$, as in Section~\ref{sec1.2.1}.

As in Section~\ref{sec1.2.2}, for $\bX,\bY,\bZ, \dotsc\in\mathrm{C}(W,\mathbb{R}^{d\times d})$, by $\mathscr{X}(t,w), \mathscr{Y}(t,w), \mathscr{Z}(t,w), \dotsc$ we mean the corresponding linear cocycles driven by the same semiflow $\varphi\colon(t,w)\mapsto t{\bcdot}w$.

For convenience we restate Theorem~\ref{thm1.9} as follows, which is the most key step toward proving Theorem~\ref{thm1.8} and which is motivated by \cite[Theorem~A]{Dai-JMSJ}.

\begin{theorem}\label{thm3.2R}
If $\bX\in\mathrm{C}(W,\mathbb{R}^{d\times d})$ obeys the robust periodic stability property driven by $\varphi$, then there are
constants $\bta<0$ and $\bT>0$ such that:
if $w\in\mathcal {P}(\varphi)$ has the prime period
$\pi_{w}\ge\bT$ and
$0=t_0<t_1<\cdots<t_\ell=\pi_w$, $\ell\ge 1$,
is a subdivision of the interval $[0,\pi_w]$ satisfying $t_k-t_{k-1}\ge\bT$ for
$k=1,\ldots,\ell$, then
\begin{equation*}
\frac{1}{\pi_w}\sum_{k=1}^\ell\log\pmb{\|}\mathscr{X}(t_k-t_{k-1},{t_{k-1}}{\bcdot}w)\pmb{\|}\le\bta
\end{equation*}
and $\rho(\mathscr{X}(\pi_w,w))\le\exp(\pi_w\bta)$.
\end{theorem}

\begin{proof}
Let $\varepsilon>0$ be a constant such that if $\bY\in\mathrm{C}(W,\mathbb{R}^{d\times d})$ satisfies $\pmb{\|}\bX-\bY\pmb{\|}<\varepsilon$ then for any $w\in\mathcal{P}(\varphi)$ with period $\pi_w>0$, we have $\rho(\mathscr{Y}(\pi_w,w))<1$, where $\rho(\bcdot)$ denotes the spectral radius of a $d\times d$ matrix as in Section~\ref{sec1.1}.

For $w\in W$ and $T>0$, write $[0,T]{\bcdot}w=\{t{\bcdot}w\,|\, 0\le t\le T\}$, i.e., the closed orbit arc of $\varphi$ connecting $w$ and $T\bcdot w$.

Using the Urysohn-Tietze extension theorem, we can find a constant $\delta=\delta(\varepsilon)>0$ with $\delta<\varepsilon$ such that for any $T>0$ and any $w\in W$, if
\begin{equation*}
\bP\colon[0,T]{\bcdot}w\rightarrow\mathbb{R}^{d\times d}\quad \textrm{such that}\quad\pmb{\|}\bP(w^\prime)\pmb{\|}\le\delta\;\forall w^\prime\in[0,T]{\bcdot}w
\end{equation*}
is continuous on the closed orbit arc $[0,T]{\bcdot}w$, then there exists a $\bY\in\mathrm{C}(W,\mathbb{R}^{d\times d})$ such that
\begin{equation*}
\pmb{\|}\bX-\bY\pmb{\|}<\varepsilon\quad \textrm{and}\quad \bY(w^\prime)=\bX(w^\prime)+\bP(w^\prime)\;\forall w^\prime\in[0,T]{\bcdot}w.
\end{equation*}
And now we put
\begin{equation*}
\varrho=\frac{\min\{\delta,1\}}{4}\quad\textrm{and}\quad \ba_*=\max_{w\in W}\pmb{\|}\bX(w)\pmb{\|}.
\end{equation*}
Let the constants $\lambda, \lambda_*, \bar{T}$ and $\bT$ be defined as follows:
\begin{equation*}
\lambda=\frac{\varrho}{4\exp(2\ba_*)},\quad\lambda_*=\frac{\lambda}{2}\exp(-\varrho/2),\quad\bar{T}=\frac{32}{\lambda\varrho}\log\frac{32}{\lambda_*^2},
\end{equation*}
and
\begin{equation}\label{eq3.1R}
\bT=\max\left\{\frac{16\ba_*\bar{T}}{\varrho},\
2\lambda\bar{T}+\frac{64}{\varrho}\log\frac{2}{\lambda_*},\
\bar{T}+2\right\},
\end{equation}
It is easy to see that $\bT$ is independent of the choice of the periodic points $w\in\mathcal{P}(\varphi)$.

Let $\w\in\mathcal{P}(\varphi)$ be arbitrarily given with the large
prime period $\pi_{\w}\gg\bT$ and let
\begin{equation}\label{eq3.2R}
0=t_0<t_1<\cdots<t_\ell=\pi_{\w},\quad \textrm{where }\ell\ge 1,\quad t_k-t_{k-1}\ge\bT
\end{equation}
be an arbitrary subdivision of the interval $[0,\pi_{\w}]$. We claim that
\begin{subequations}\label{eq3.3R}
\begin{align}
&\frac{1}{\pi_{\w}}\sum_{k=1}^\ell\log\pmb{\|}\mathscr{X}(t_k-t_{k-1},{t_{k-1}}{\bcdot}\w)\pmb{\|}\le
-\frac{\varrho}{4},\\
\intertext{or equivalently,}&\varDelta:=\prod_{k=1}^\ell\pmb{\|}\mathscr{X}(t_k-t_{k-1},{t_{k-1}}{\bcdot}\w)\pmb{\|}\le\exp(-\frac{\varrho}{4}\pi_{\w}),
\end{align}\end{subequations}
which completes the proof of Theorem~\ref{thm3.2R}.

Indeed, to apply Lemma~\ref{lem3.1R}, we first take and then fix some vector $y_\natural\in\mathbb{R}^d$ such that
\begin{equation*}
\|\mathscr{X}(\pi_{\w},\w)y_\natural\|=1\quad\textrm{and}\quad\|\mathscr{X}(t_{\ell-1},\w)y_\natural\|=\inf_{x\in\mathbb{R}^d, \|\mathscr{X}(\pi_{\w},\w)x\|=1}\|\mathscr{X}(t_{\ell-1},\w)x\|.
\end{equation*}
Put
\begin{equation}\label{eq3.4R}
y_0=\mathscr{X}(\pi_{\w},\w)y_\natural.
\end{equation}
Then by applying Lemma~\ref{lem3.1R} with $a(t)=\bX(t{\bcdot}\w)$ and $T=\pi_{\w}$, one can find a linear equation
\begin{equation}\label{eq3.5R}
\frac{dy}{dt}=\big{(}\bX(t{\bcdot}\w)+B_{\w}(t)\big{)}y\quad(t\in\mathbb{R}_+\textrm{ and } y\in\mathbb{R}^d)
\end{equation}
such that $B_{\w}(t)$ is continuous in $t$ and
\begin{equation}\label{eq3.6R}
B_{\w}(t)|_{\{0\}\cup[\pi_{\w}-\frac{1}{8},\infty)}\equiv0,\quad\sup_{t\ge0}\pmb{\|}B_{\w}(t)\pmb{\|}<\varrho.
\end{equation}
Note that by Lemma~\ref{lem3.1R} (\ref{eq3.5R}) has a solution $y_{\w}(t)$ such that
\begin{equation}\label{eq3.7R}
y_{\w}(0)=y_0\quad\textrm{and}\quad y_{\w}(\pi_{\w})=\|y_{\w}(\pi_{\w})\|y_0\;\textrm{ or }y_{\w}(\pi_{\w})=-\|y_{\w}(\pi_{\w})\|y_0,
\end{equation}
and
\begin{equation}\label{eq3.8R}
\|y_{\w}(\pi_{\w})\|=\prod_{k=1}^\ell\sup_{x\in\mathbb{R}^d,\|\mathscr{X}(t_{k-1},\w)x\|=1}\|\mathscr{X}(t_k,\w)x\|=\varDelta.
\end{equation}
Next we will prove (\ref{eq3.3R}) under condition (\ref{eq3.2R}).

On the contrary, assume (\ref{eq3.3R}) is not true. So $-\frac{1}{\pi_{\w}}\log\varDelta<\frac{\varrho}{4}$.  For any $0\le t\le \pi_{\w}$, we put
\begin{equation*}
\bP(t{\bcdot}\w)=B_{\w}(t)+ \alpha I_d,\quad \textrm{where }
\alpha:=\begin{cases}-\frac{1}{\pi_{\w}}\log\varDelta\in[0,{\varrho}/4{})& \textrm{if }\varDelta\le1,\\
0& \textrm{if }\varDelta>1.
\end{cases}
\end{equation*}
Since $\pmb{\|}\bP\pmb{\|}<2\varrho<\delta$, there exists a $\bY\in\mathrm{C}(W,\mathbb{R}^{d\times d})$ such that
\begin{equation*}
\pmb{\|}\bX-\bY\pmb{\|}<\varepsilon\quad \textrm{and}\quad \bY(t{\bcdot}\w)=\bX(t{\bcdot}\w)+\bP(t{\bcdot}\w)\;\forall t\in[0,\pi_{\w}],
\end{equation*}
and that
\begin{equation}
\hat{y}(t)=y_{\w}(t)\exp(\alpha t),\quad \textrm{where }0\le t\le\pi_{\w},
\end{equation}
is a solution of the induced equation
\begin{equation}
\frac{dy}{dt}=\bY(t{\bcdot}\w)y\quad(0\le t\le\pi_{\w}\textrm{ and } y\in\mathbb{R}^d).
\end{equation}
 Thus according to (\ref{eq3.8R}) and (\ref{eq3.7R}), we can obtain that
\begin{equation*}
\|y_{\w}(\pi_{\w})\|>\exp(-\varrho\pi_{\w}/4)\quad \textrm{and hence}\quad
\|\hat{y}(\pi_{\w})\|=\|y_0\|\textrm{ if }\varDelta\le 1\textrm{ and }\|\hat{y}(\pi_{\w})\|>\|y_0\|\textrm{ if }\varDelta>1.
\end{equation*}
Since $\hat{y}(0)=y_0$ and $\hat{y}(\pi_{\w})\in\left\{\pm y_0\|y_{\w}(\pi_{\w})\|\exp(\alpha\pi_{\w})\right\}$, we see that $\rho(\mathscr{Y}(\pi_{\w},\w))\ge1$, it is a contradiction to the robust periodic stability property of $\bX$.

This proves (\ref{eq3.3R}) and hence Theorem~\ref{thm3.2R} by letting $\bta=-\varrho/4$.
\end{proof}

\subsection{Uniform stability on periodic orbits with uniformly bounded periods}\label{sec3.3}
We have considered the quasi stability of periodic points with large prime periods. We now will study the periodic points with small prime periods.

For any $T>0$, let $\mathcal{P}_T(\varphi)$ be the set of all periodic points whose prime periods are less than or equal to $T$.

Since $\mathcal{P}_T(\varphi)$ is a compact invariant set of $\theta$ and every ergodic probability measure in $\mathcal{P}_T(\varphi)$ is supported on a periodic orbit, the following lemma follows easily from the semi-uniform subadditive ergodic theorem (see, e.g.,~\cite{Sch, SS, Cao, Dai-Non}).

\begin{lemma}\label{lem3.3R}
Let $\bX\in\mathrm{C}(W,\mathbb{R}^{d\times d})$ obey the robust periodic stability property driven by $\theta$. Then for any $T>0$, there are
constants $C>0$ and $0<\gamma<1$ such that
\begin{equation*}
\pmb{\|}\mathscr{X}(t,p)\pmb{\|}\le C\gamma^t\quad\forall t\ge0,
\end{equation*}
for every $p\in\mathcal{P}_T(\varphi)$.
\end{lemma}

In the differentiable dynamical systems case, $\mathcal{P}_T(\varphi)$ consists of at most finite number of periodic orbits. This is not the case under our situation.

\subsection{An ergodic lemma}\label{sec3.4}
Let $\mathscr{X}\colon\mathbb{R}_+\times W\rightarrow\mathrm{GL}(d,\mathbb{R})$ be the linear cocycle driven by $\varphi\colon(t,w)\mapsto t{\bcdot}w$, for an arbitrary $\bX\in\mathrm{C}(W,\mathbb{R}^{d\times d})$. For any $T>0$, we define a qualitative function
\begin{equation*}
\xi_T\colon W\rightarrow\mathbb{R};\quad w\mapsto\frac{1}{T}\log\pmb{\|}\mathscr{X}(T,w)\pmb{\|}.
\end{equation*}
It is a continuous function of $w\in W$, since $\varphi$ and $\bX$ both are continuous.

The following criterion for almost sure stability is useful for proving Theorem~\ref{thm1.8} later.

\begin{lemma}\label{lem3.4R}
Let $\mu$ be an arbitrary ergodic probability measure of $\varphi$ on $W$. If
\begin{equation*}
\int_W\xi_T(w)\mu(dw)<0\quad \textrm{for some }T>0,
\end{equation*}
then $\bX$ at $\mu$ has a negative maximal Lyapunov exponent; i.e.,
\begin{equation*}
\lim_{t\to\infty}\frac{1}{t}\log\pmb{\|}\mathscr{X}(t,w)\pmb{\|}<0
\end{equation*}
for $\mu$-a.e. $w\in W$.
\end{lemma}

\begin{proof}
According to the Oselede\v{c} multiplicative ergodic theorem, there exists a constant $\chi\in\mathbb{R}$ such that
\begin{equation*}
\chi=\lim_{t\to\infty}\frac{1}{t}\log\pmb{\|}\mathscr{X}(t,w)\pmb{\|}
\end{equation*}
for $\mu$-a.e. $w\in W$. Let
\begin{equation*}
f_T\colon W\rightarrow W;\quad w\mapsto T{\bcdot}w.
\end{equation*}
Then $f_T$ is a continuous transformation of $W$, which preserves $\mu$, but not necessarily ergodic.
From the ergodic decomposition theorem, it follows that there exists a family of $f_T$-ergodic probability measures $\{\mu_w\}_{w\in W}$ on the Borel measurable space $(W,\mathscr{B})$ such that
\begin{equation*}
\int_W\psi(w)\mu(dw)=\int_W\left(\int_W\psi(w^\prime)\mu_w(dw^\prime)\right)\mu(dw)\quad\forall \psi\in \mathscr{L}^1(W,\mathscr{B},\mu).
\end{equation*}
Since $\int_W\xi_T(w)\mu(dw)<0$ and $\xi_T\in \mathscr{L}^1(W,\mathscr{B},\mu)$, we can find a Borel set $\Lambda\subset W$ with $\mu(\Lambda)>0$ such that
\begin{equation*}
\int_W\xi_T(w^\prime)\mu_w(dw^\prime)<0\quad\forall w\in\Lambda.
\end{equation*}
Then from the Birkhoff ergodic theorem for $(W,f_T,\mu_w, \xi_T)$, it follows that for any $w\in\Lambda$,
\begin{equation*}
\lim_{n\to\infty}\frac{1}{nT}\log\pmb{\|}\mathscr{X}(nT,w^{\prime\prime})\pmb{\|}\le\lim_{n\to+\infty}\frac{1}{n}\sum_{k=0}^{n-1}\xi_T(f_T^k(w^{\prime\prime}))=\int_W\xi_T(w^\prime)\mu_w(dw^\prime)<0
\end{equation*}
for $\mu_w$-a.e. $w^{\prime\prime}\in W$. This implies that $\chi<0$, as desired. This proves Lemma~\ref{lem3.4R}.
\end{proof}
\subsection{Proof of Theorem~\ref{thm1.8}}
Now we are ready to complete the proof of Theorem~\ref{thm1.8} stated in Section~\ref{sec1.2.2}.

\begin{proof}[Proof of Theorem~\ref{thm1.8}]
In order to prove Theorem~\ref{thm1.8}, it is sufficient to show that for any ergodic probability measure $\mu$ of $\varphi$ on $W$, $\bX$ at $\mu$ has the negative maximal Lyapunov exponent. From now on, let $\mu$ be an arbitrary ergodic probability measure of $\varphi$ on $W$.

If $\mu$ is either atomic or supported on a periodic orbit of $\theta$, then the statement holds from Lemma~\ref{lem3.3R} and Theorem~\ref{thm3.2R}. So, there is no loss of generality in assuming that $\mathrm{supp}(\mu)$ is neither a point nor a periodic orbit of $\varphi$. We can then choose a sequence of periodic orbits $P_n$ with the prime periods $\pi_n$ such that as $n$ tends to $+\infty$,
\begin{equation*}
\pi_n\uparrow+\infty,\quad\mu_{P_n}^{}\xrightarrow{\textrm{in the weak-$*$ topology}}\mu,\quad\textrm{and}\quad P_n\xrightarrow{\textrm{in the Hausdorff metric}}\mathrm{supp}(\mu),
\end{equation*}
where $\mu_{P_n}^{}$ denotes the unique ergodic probability measure of $\varphi$ supported on the periodic orbit $P_n$ for all $n$, as stated in Section~\ref{sec1.2}.

Let the constants $\bta<0$ and $\bT>0$ be given by Theorem~\ref{thm3.2R} and $\xi_{\bT}\colon w\mapsto\frac{1}{\bT}\log\pmb{\|}\mathscr{X}(\bT,w)\pmb{\|}$. Define the continuous transformation $f_{\bT}\colon W\rightarrow W;\, w\mapsto{\bT}{\bcdot}w$. Then, $\mu$ and $\mu_{P_n}^{}, n=1,2,\dotsc$, all are invariant probability measures of $f_{\bT}$ on $W$, but not necessarily $f_{\bT}$-ergodic.
Now according to Lemma~\ref{lem3.4R}, to prove Theorem~\ref{thm1.8} we need only prove $\int_W\xi_{\bT}(w)\mu(dw)<0$.
Since it holds that
\begin{equation*}
\lim_{n\to+\infty}\int_W\xi_{\bT}(w)\mu_{P_n}^{}(dw)=\int_W\xi_{\bT}(w)\mu(dw),
\end{equation*}
it is sufficient to find some constant $\gamma<0$ such that
\begin{equation*}
\int_W\xi_{\bT}(w)\mu_{P_n}^{}(dw)\le\gamma
\end{equation*}
as $n$ is sufficiently large.
For that, we write the prime periods
\begin{equation*}
\pi_n=\ell_n\bT+\varDelta_n\quad \textrm{where }\varDelta_n=0\textrm{ or }\bT<\varDelta_n<2\bT,\quad n=1,2,\dotsc.
\end{equation*}
Then from Theorem~\ref{thm3.2R}, for any $p\in P_n$
\begin{equation*}
\frac{\bT}{\ell_n\bT+\varDelta_n}\sum_{k=0}^{\ell_n-1}\xi_{\bT}(f_{\bT}^k(p))+\frac{1}{\pi_n}\log\pmb{\|}\mathscr{X}(\varDelta_n,(\ell_n\bT){\bcdot}p)\pmb{\|}\le\bta
\end{equation*}
for all $n\ge1$. Since $\mathscr{X}(t,w)$ is jointly continuous in $(t,w)$ and both $[\bT,2\bT], W$ are compact, there exists an $N\ge1$ such that for any $n\ge N$,
\begin{equation}\label{eq3.11R}
\frac{1}{\ell_n}\sum_{k=0}^{\ell_n-1}\xi_{\bT}(f_{\bT}^k(p))\le\frac{\bta}{2}\quad \forall p\in P_n.
\end{equation}
On the other hand, from the Birkhoff ergodic theorem, for every $n\ge1$ we have
\begin{equation*}
\lim_{m\to+\infty}\frac{1}{m}\sum_{k=0}^{m-1}\xi_{\bT}(f_{\bT}^k(p))=\bar{\xi}_{\bT}^{(n)}(p),\quad\mu_{P_n}^{}\textrm{-a.e. }p\in P_n;
\end{equation*}
and
\begin{equation}\label{eq3.12R}
\int_W\bar{\xi}_{\bT}^{(n)}d\mu_{P_n}^{}=\int_W\xi_{\bT}d\mu_{P_n}^{}.
\end{equation}
Therefore by (\ref{eq3.11R}), as $n\ge N$ we have
\begin{equation*}
\bar{\xi}_{\bT}^{(n)}(p)=\lim_{m\to+\infty}\frac{1}{m\ell_n}\sum_{k=0}^{m\ell_n-1}\xi_{\bT}(f_{\bT}^k(p))=\lim_{m\to+\infty}\frac{1}{m}\sum_{i=0}^{m-1}\left\{\frac{1}{\ell_n}\sum_{k=0}^{\ell_n-1}\xi_{\bT}(f_{\bT}^{k+i\ell_n}(p))\right\}\le\frac{\bta}{2}
\end{equation*}
for $\mu_{P_n}^{}$-a.e. $p\in P_n$; noting here that $f_{\bT}^{i\ell_n}(p)\in P_n$ for all $p\in P_n$ and any integer $i\ge0$ for $P_n$ is $\theta$-invariant. So by the  equality (\ref{eq3.12R}), we can obtain that
\begin{equation*}
\int_W\xi_{\bT}d\mu_{P_n}^{}\le\frac{\bta}{2}\quad \textrm{as }n\ge N, \quad\textrm{and then}\quad\int_W\xi_{\bT}d\mu<0.
\end{equation*}
This proves the statement of Theorem~\ref{thm1.8}.
\end{proof}

We note that the proof of Theorem~\ref{thm1.8} is a modification of that of Liao~\cite[Theorem~3.1]{Liao-CAM}. Since $\bX$ does not need to be H\"{o}lder continuous and the closing by periodic orbits property of $\varphi$ is not necessarily exponential as in \cite{Dai-FM,Kal}, there is no a general approximation theorem of Lyapunov exponents as \cite[Theorem~1.3]{Dai-FM} and \cite{Kal} applicable here. So, the proof of Theorem~\ref{thm1.8} presented here is itself of interest.

\section{Stability of linear cocycles driven by irrational rotations}\label{sec4}
In this section, we will study linear cocycles driven by an irrational rotation on $\mathbb{T}^1$. For our convenience, we introduce a metric on the product topological space $\mathbb{R}\times\mathbb{T}^1$ as follows: for any $z=(y,e^{2\pi\mathfrak{i}x})\in\mathbb{R}\times\mathbb{T}^1$ and $z^\prime=(y^\prime,e^{2\pi\mathfrak{i}x^\prime})\in\mathbb{R}\times\mathbb{T}^1$ where $x,x^\prime[0,1)$, set $d(z,z^\prime)=|y-y^\prime|+|x-x^\prime|$ mod 1. It is easy to see that $d(\cdot,\cdot)$ is a metric compatible with the product topology of $\mathbb{R}\times\mathbb{T}^1$.

\subsection{Proof of Theorem~\ref{thm1.11}}\label{sec4.1}

We will prove Theorem~\ref{thm1.11} stated in Section~\ref{sec1.3} using the quasi contraction lemma (i.e. Theorem~\ref{thm1.9}) and Theorem~\ref{thm1.8} proved in Section~\ref{sec3}.

\begin{proof}[Proof of Theorem~\ref{thm1.11}]
Let $\omega\in(0,1)$ be an irrational number and $\bS\in\mathrm{C}(\mathbb{T}^1,\mathrm{GL}(d,\mathbb{R}))$ given as in Theorem~\ref{thm1.11}. For $\omega$, from Dirichlet's theorem we choose a sequence of positive integer pairs $(p_n,q_n)_{n\ge1}$ such that $0<p_n/q_n<1$ for all $n\ge1$.

Since $p_n/q_n\to\omega$ as $n\to\infty$ and $[0,1]\times\mathbb{T}^1$ is compact under the standard product topology, we see that
\begin{equation*}
W:=\left(\{p_n/q_n\colon n\ge1\}\cup\{\omega\}\right)\times\mathbb{T}^1,
\end{equation*}
as a closed subspace of $[0,1]\times\mathbb{T}^1$, is also a compact metric space. We simply write $\mathbb{T}_y^1=\{y\}\times\mathbb{T}^1$ for all $y\in[0,1]$.
We now extend $\bS$ from $\mathbb{T}_\omega^1$ onto $W$ as follows:
\begin{equation*}
\widetilde{\bS}\colon W\rightarrow\mathrm{GL}(d,\mathbb{R})\quad \textrm{by }(y,z)\mapsto S(z)\;\forall (y,z)\in W.
\end{equation*}
Moreover, we define the homeomorphism from $W$ onto itself
\begin{equation*}
T\colon W\rightarrow W\quad \textrm{by }T(y,z)=(y,R_{y}(z))\;\forall (y,z)\in W,
\end{equation*}
where $R_y\colon \mathbb{T}_y^1\rightarrow\mathbb{T}_y^1$ is as in (\ref{eq1.9}) with $y$ instead of $x$. For any $z\in\mathbb{T}^1$, since
\begin{equation*}
\big{|}\omega-\frac{p_n}{q_n}\big{|}<\frac{1}{q_n^2}\quad\textrm{and}\quad T^{q_n}(\frac{p_n}{q_n},z)=(\frac{p_n}{q_n},z)\quad\forall n\ge1,
\end{equation*}
we can obtain that
\begin{equation*}
\big{|}T^k(\omega,z)-T^k(\frac{p_n}{q_n},z)\big{|}<\frac{1}{q_n}\cdot|z|\quad\forall k=0,1,\dotsc,q_n-1.
\end{equation*}
This implies that $(W,T)$ has the closing by periodic orbits property (cf.~Definition~\ref{def1.12}).

On the other hand, from $\widetilde{\bS}$ there is a naturally defined linear cocycle
\begin{equation*}
\widetilde{\mathscr{S}}\colon \mathbb{N}\times W\rightarrow\mathrm{GL}(d,\mathbb{R});\quad (n,(y,z))\mapsto S(R_{y}^{n-1}(z))\dotsm S(R_y(z))S(z)
\end{equation*}
driven by $T$. We note that although the robust periodic stability of $\bS$ on $\mathbb{T}^1$ defined by Definition~\ref{def1.10} cannot induce the robust periodic stability of $\widetilde{\bS}$ on $W$ defined by discretization of Definition~\ref{def1.7}, yet from the product structure of $W$, the distribution of periodic orbits of $T$ and from the proof of Theorem~\ref{thm3.2R}, it follows that $\widetilde{\mathscr{S}}$ has the quasi contraction property on the periodic points of $T$.

Therefore, $\widetilde{\mathscr{S}}$ driven by $T$ is uniformly exponentially stable from Theorem~\ref{thm1.8}. This implies that $\mathscr{S}_\omega$ driven by $R_\omega$ is uniformly exponentially stable too, as desired.

This completes the proof of Theorem~\ref{thm1.11}.
\end{proof}

We note that since $p_n/q_n\to\omega$ as $n\to\infty$, $\{p_n/q_n\,|\,n\ge1\}\cup\{\omega\}$ is a compact subset of $\mathbb{R}$ and $W$ is compact. However, in the definition of the driving dynamical system $T$, $p_n/q_n$ and $\omega$ all play the role of the rotation numbers. The same points should be noted in the proof of Theorem~\ref{thm4.1} later.
\subsection{Complete periodic stability does not imply the uniform stability}\label{sec4.2}
Now we will construct examples which show that the complete periodic stability condition (\ref{eq1.3}) does not need to imply the unform stability even in the $1$-dimensional case in our context.

\begin{theorem}\label{thm4.1}
Let $\omega\in(0,1)$ be an irrational number and $\gamma$ an arbitrary constant with $0<\gamma<1$. Let $\{p_n/q_n\}_{n\ge1}$ be a sequence of rational numbers such that $q_n\uparrow\infty$ and $|\omega-\frac{p_n}{q_n}|<\frac{1}{q_n^2}$.
Define a continuous $1\times 1$ matrix-valued function
\begin{equation*}
\bS\colon W:=\left(\{p_n/q_n\,|\,n=1,2,\dotsc\}\cup\{\omega\}\right)\times\mathbb{T}^1\rightarrow(0,1]\subset\mathbb{R}^{1\times 1}
\end{equation*}
by
\begin{equation*}
\bS|\{p_n/q_n\}\times\mathbb{T}^1\equiv\gamma^{1/q_n}\quad\textrm{and}\quad\bS|\{\omega\}\times\mathbb{T}^1\equiv1.
\end{equation*}
Then, driven by $T\colon (y,z)\mapsto (y,R_y(z))$ which has the closing by periodic orbits property, $\bS$ is completely periodic stable, but not uniformly stable.
\end{theorem}

\begin{proof}
Indeed, for any periodic point $(p_n/q_n,z)\in W$ of $T$, its period is $q_n$ under the iteration of $T$ and the induced linear cocycle has the spectral radius $\gamma$ over this corresponding periodic orbit. However, the induced linear cocycle has Lyapunov exponent zero at every aperiodic point $(\omega,z)\in W$ of $T$. This completes the proof of Theorem~\ref{thm4.1} from Theorem~\ref{thm1.11}.
\end{proof}

For a continuous transformation $T\colon W\rightarrow W$ on a compact metric space $(W,\bd)$, a point $w\in W$ is said to be \textit{nonwandering} with respect to $T$ if for every neighborhood $U$ of $w$, there exists an $n\ge1$ with $U\cap T^{-n}U\not=\varnothing$. Equivalently, a point $w$ is nonwandering if and only if for any $\varepsilon>0$ there is a point $y\in W$ and an $n\ge1$ such that $\bd(w,y)<\varepsilon$ and $\bd(w,T^n(y))<\varepsilon$. The set of all nonwandering points of $T$ is called the \textit{nonwandering set} of $T$ and denoted by $\Omega_T(W)$. Clearly, $\Omega_T(W)$ is an $T$-invariant closed subset of $W$ such that $\mu(\Omega_T(W))=1$ for every $T$-invariant probability measure $\mu$ on $W$. Hence from \cite[Corollary~2.7]{Dai-LAA}, it follows that for any continuous matrix-valued function
$\bC\colon W\rightarrow \mathbb{R}^{d\times d}$ by $w\mapsto C(w)$,
the joint spectral radius $\brho(\bC,T)$ defined as in Definition~\ref{def1.13} is such that
\begin{equation*}
\brho(\bC,T)=\lim_{n\to\infty}\max_{w\in \Omega_T(W)}\sqrt[n]{\pmb{\|}\mathscr{C}(n,w)\pmb{\|}}=\inf_{n\ge1}\max_{w\in\Omega_T(W)}\sqrt[n]{\pmb{\|}\mathscr{C}(n,w)\pmb{\|}}.
\end{equation*}
It is easy to check that if $T$ satisfies the closing by periodic orbits property described as in Definition~\ref{def1.12}, then the periodic points of $T$ are dense in $\Omega_T(W)$. By $W_{\textrm{per}}^n(T)$ we mean the set of all periodic points of period $n$ for all $n\ge1$. Write $W_\textrm{per}(T)=\bigsqcup_{n\ge1}W_\textrm{per}^n(T)$.

\begin{remark}\label{rem4.2}
Theorem~\ref{thm4.1} not only may be served as a counterexample to the spectral finiteness mentioned in Section~\ref{sec1.4}, but also resists \cite[Question~3]{Dai-LAA}. However, Theorem~\ref{thm4.1} is not a counterexample to \cite[Question~3]{Dai-LAA}, since the complete periodic stability condition, that is described as
\begin{equation*}
\exists\,\bgamma<1\textrm{ such that }\rho(C(T^{n-1}w)\dotsm C(w))\le\bgamma\;\forall w\in W_\textrm{per}^n(T)\textrm{ and }n\ge1,
\end{equation*}
is weaker than the following condition which is required by \cite[Question~3]{Dai-LAA}:
\begin{equation*}
\exists\,\bgamma<1\textrm{ and }N\ge1\textrm{ such that }\rho(C(T^{n-1}w)\dotsm C(w))\le\bgamma\;\forall w\in W_\textrm{per}(T)\textrm{ and }n\ge N;
\end{equation*}
since the period of $w\in W_{\textrm{per}}(T)$ is not necessarily equal to $n$.
\end{remark}

\section{Concluding remarks}
We have mainly proven that for a linear cocycle driven by a dynamical system having the closing by periodic orbits property, it is uniformly exponentially stable if and only if it is robustly periodically stable (Theorems~\ref{thm1.2}, \ref{thm1.8} and \ref{thm1.11}). We note here that we have not imposed any additional conditions on the Markov transition matrix $\bP$ for the Markovian chain $\bxi$ which is equivalent to the MJLS (\ref{eq1.1}). If $\bP$ is irreducible and aperiodic (i.e. there is some $N>0$ such that $\bP^N$ is strictly positive), then the $N$-fold iteration of the finite-type subshift $\theta_{\mathbb{A}}^+$ is equivalent a full shift and further the complete periodic stability condition (\ref{eq1.3}) implies the uniform stability of the MJLS (\ref{eq1.1}). Hence our results and methods of proof should be useful for the stability analysis of Markovian jump linear systems.

In the classical case of the stability analysis of switched linear dynamical system, a powerful tool is the Elsner reduction theorem; in other words, for any irreducible
$\A=\{A_1,\dotsc,A_K\}\subset\mathbb{R}^{d\times d}$ with the joint spectral radius $\brho(\A)>0$, there always exists a vector norm $\|\cdot\|_*$ on $\mathbb{R}^d$ such that
$$
\|A_{i_n}\dotsm A_{i_1}x\|_*\le\brho(\A)^n\|x\|_*,\quad\forall x\in\mathbb{R}^d, n\ge1\textrm{ and }(i_1,\dotsc,i_n)\in\{1,\dotsc,K\}^n.
$$
However, this is not the case in our situation. To get around those obstruction points that are overcame by Elsner's reduction theorem in the classical case, we have employed Shantao Liao's methods established for the theory of differentiable dynamical systems and proved a Gel'fand-Berger-Wang formula of MJLS using approximation of Lyapunov exponents by periodic orbits.

In addition, counterexamples have been constructed to the robustness condition (Theorem~\ref{thm4.1}) and to the spectral finiteness of linear cocycle (Theorem~\ref{thm1.14}).

\section*{\textbf{Acknowledgments}}%
The author would like to thank Professors Victor~Kozyakin, Yu~Huang and Mingqing~Xiao for some helpful discussion. Particularly he is grateful to the anonymous reviewers for their comments.

Project was supported partly by National Natural Science Foundation of China (Nos. 11071112 and 11271183) and PAPD of Jiangsu Higher Education Institutions.
\bibliographystyle{amsplain}

\end{document}